\documentclass{svjour3}                     

\smartqed  
\usepackage{graphicx}
\usepackage{amsmath}
\usepackage{amssymb}
\usepackage{epsfig}
\usepackage{algorithm}
\usepackage{epstopdf}
\usepackage{subfigure}
\usepackage{multirow}
\numberwithin{equation}{section}
\usepackage [pdftex,colorlinks,anchorcolor=blue, citecolor=green,linkcolor=red]{hyperref}

\begin{document}
\newtheorem{Thm}{Theorem}[section]
\newtheorem{Cor}[Thm]{Corollary}
\newtheorem{Lem}[Thm]{Lemma}
\newtheorem{Prop}[Thm]{Proposition}
\newtheorem{Def}[Thm]{Definition}
\newtheorem{Rem}[Thm]{Remark}
\newtheorem{Assump}[Thm]{Assumption}
\newtheorem{Exam}[Thm]{Example}

\title{Heavy-ball-based optimal thresholding algorithms for  sparse linear inverse problems \thanks{The work was founded by the National Natural Science Foundation of China
(NSFC 12071307 and 11771255), Young Innovation Teams of Shandong
Province (2019KJI013), and  Shandong Province Natural Science Foundation
(ZR2021MA066, ZR2023MA020).}}


\titlerunning{Optimal $k$-thresholding algorithms}        

\author{Zhong-Feng Sun \and Jin-Chuan Zhou \and  Yun-Bin Zhao}


\institute{Zhong-Feng Sun\at
               School of Mathematics and Statistics, Shandong
University of Technology, Zibo, Shandong,  China. \email{zfsun@sdut.edu.cn}
 \and  Jin-Chuan Zhou\at  School of Mathematics and Statistics, Shandong
University of Technology, Zibo, Shandong,  China. \email{jinchuanzhou@sdut.edu.cn}
\and Yun-Bin Zhao \at Corresponding author. Shenzhen Research Institute of Big Data, Chinese University of Hong Kong, Shenzhen,  Guangdong, China. \email{yunbinzhao@cuhk.edu.cn}
}

\date{Received: date / Accepted: date}

\maketitle

\begin{abstract}
 Linear inverse problems arise in diverse engineering fields especially in signal and image reconstruction. The development of computational methods for linear inverse problems with sparsity is one of the recent trends in this field. The so-called optimal $k$-thresholding is a newly introduced method for sparse optimization and linear inverse problems. Compared to other sparsity-aware algorithms, the advantage of optimal $k$-thresholding method lies in that it performs thresholding and error metric reduction simultaneously and thus works stably and robustly for solving medium-sized linear inverse problems. However, the runtime of this method is generally high  when the size of the problem is large. The purpose of this paper is to propose an acceleration strategy for this method. Specifically, we propose a heavy-ball-based optimal $k$-thresholding (HBOT) algorithm and its relaxed variants for sparse linear inverse problems. The convergence of these algorithms is shown under the restricted isometry property. In addition, the numerical performance of the heavy-ball-based relaxed optimal $k$-thresholding pursuit (HBROTP) has been evaluated, and simulations indicate that HBROTP admits robustness for signal and image reconstruction even in noisy environments.
\end{abstract}



\section{Introduction}\label{sec-inct}
In recent years, the linear inverse problem has
gained much attention in various fields such as wireless
communication \cite{BSR17,CZJAZ21} and signal/image processing \cite{BT09,BD10,E10,E12,LAZS20,OS18,TR20,XDY21}. A typical linear inverse problem is about the reconstruction of   unknown data  $z\in \mathbb{R}^{r}$ from the acquired linear measurements
\begin{equation}\label{model-lip}
y=\Phi z+\nu,
\end{equation}
where $\Phi\in \mathbb{R}^{m\times r}$ is a given measurement matrix, $y\in \mathbb{R}^{m}$ are the acquired measurements, and $\nu\in \mathbb{R}^{m}$ are the measurement errors. In this paper, we consider the case $m< r,$ for which   it is generally impossible to reconstruct the data $z$ from  the linear system \eqref{model-lip} unless $z$ possesses a certain   structure such as sparsity. Fortunately,  in many practical applications, the signal to recover possesses certain sparse structure or it can be sparsely represented under a suitable transformation.  For instance,  many natural image can be sparsely represented via wavelet transforms. Suppose that  $z$ can be sparsely represented via the basis $\Psi\in \mathbb{R}^{r\times n}$ ($r\leq n$), i.e.,  $z=\Psi x$ where the vector $x\in \mathbb{R}^{n}$ is either $k$-sparse or $k$-compressible for some integer number $k \ll n.$ A vector is said to be $k$-sparse if $\| x\|_0\leq k$, and $k$-compressible if $x$ can be approximated by a $k$-sparse vector, where $\| \cdot\|_0$ denotes the number of nonzero entries of a vector.  With a sparse representation of $z,$ the model \eqref{model-lip} can be written as
\begin{equation}\label{model-slip}
y=Ax+\nu,
\end{equation}
where  $A=\Phi\Psi\in \mathbb{R}^{m\times n}$ ($m < n$) is still referred to as a measurement matrix. In this case, the problem (\ref{model-lip}) is transformed to the so-called sparse linear inverse (SLI) problem which is to reconstruct the sparse data $x$ via the linear system (\ref{model-slip}). Once the sparse data $x$ is reconstructed,  the original data $z$ can be immediately obtained by setting $ z:= \Psi x.$  The SLI problem can be formulated as an optimization problem (see, e.g., \cite{BT09,BD10,E10,FR13,CZJAZ21,OS18,TW10}). Typically, it can be formulated as the sparse optimization problem
 \begin{equation}\label{model}
\underset{u}{\rm min}\{ {\| y-Au\|}_2^2: \| u\|_0\leq k\}.
\end{equation}
It can also be formulated as the $\ell_1$-minimization (basis pursuit) problem
 \begin{equation}\label{L1-min}
\underset{u}{\rm min}\{ \| u\|_1: Au=y\}
\end{equation}
  as well as the LASSO problem
 \begin{equation}\label{Lasso}
\underset{u}{\rm min} {\| y-Au\|}_2^2+ \mu\| u\|_1,
\end{equation}
where  $\mu>0$ is called a regularization parameter. All these models,  (\ref{model})-(\ref{Lasso}), are widely used in signal and image reconstruction with sparsity.

Depending on the problem formulations, several classes of algorithms for SLI problems have been developed over the past decades, including the thresholding algorithms \cite{E10,E12,FR13}, greedy methods \cite{DM09,NT09,TG07},  convex optimization \cite{CT05,CWB08,CDS98,ZL17}, nonconvex optimization   \cite{C07}, and Bayesian  methods  \cite{SPZ08,WB04}. In this paper, we focus on the model (\ref{model}) for which the thresholding algorithms are particularly convenient to develop. The thresholding method was first proposed by Donoho and  Johnstone \cite{DJ94}. It has experienced a significant development since 1994 and has evolved into a large family of algorithms which includes the hard thresholding \cite{BTW15,BD08,BD10,F11,KC14,MZ20}, soft thresholding \cite{DDM04,D95,E06} and optimal $k$-thresholding algorithms \cite{Z20,ZL21}. It is worth stressing that an advantage of thresholding methods  is that the algorithms can guarantee the generated points being feasible to the problem \eqref{model}. The simplest thresholding method might be the  iterative hard thresholding (IHT) \cite{BD08}. The combination of IHT and orthogonal projection yields the hard thresholding pursuit (HTP) \cite{F11}. Due to low computational complexity, IHT and HTP  have been widely used in   signal reconstruction with compressive samplings   \cite{BTW15,B12,BD10,KC14}.

However, as pointed out in \cite{Z20,ZL21}, performing hard thresholding on non-sparse iterates may not necessarily reduce the objective  value of  (\ref{model}) and thus may cause numerical oscillation during iterations. Thus the optimal $k$-thresholding operator was introduced in \cite{Z20} (see also  \cite{ZL21}) to alleviate such a weakness of hard thresholding.  This operator may perform thresholding on iterates and, in the meantime,  reduce the objective value of (\ref{model}). The optimal $k$-thresholding (OT)   and optimal $k$-thresholding pursuit (OTP) algorithms are first developed in \cite{Z20}. Recall that the optimal $k$-thresholding of a given vector $v\in \mathbb{R}^n $ is defined as
\begin{equation}\label{OT-model}
\underset{w}{\rm min} \{ {\| y-A(v\circ{w})\|}_2^2: ~  { \bf e}^T w=k, ~ w \in\{0,1\}^n\},
\end{equation}
where  ${ \bf e}=(1,1,\ldots,1)^T\in\mathbb{R}^n$,  $\{0,1\}^n$ is the set of $n$-dimensional binary vectors and  $v\circ w:=(v_1w_1,\ldots, v_nw_n)^T$ denotes the Hadamard product of two vectors. However, from a computational point of view,  it is generally more convenient to solve the following convex optimization
\begin{equation}\label{ROT-model}
\underset{w}{\rm min} \{ {\| y-A(v\circ{w})\|}_2^2: \ \ { \bf e}^T w=k, \ \ 0\leq w \leq \bf e\},
\end{equation} which is  a tight relaxation of \eqref{OT-model}.
  This problem is referred to as data compressing problem in \cite{Z20,ZL21}. Based on \eqref{ROT-model},  the relaxed optimal $k$-thresholding (ROT$\omega$) and relaxed  optimal $k$-thresholding pursuit (ROTP$\omega$) algorithms were proposed in \cite{Z20,ZL21}, where $\omega$  represents times of data compression that are performed in the algorithms. When $\omega=1$, the algorithm is termed as ROTP. Some modifications of ROTP using partial gradient and Newton-type search direction were studied recently in \cite{MZ22,MZKS22}.  While the convex optimization  \eqref{ROT-model} can be efficiently solved by existing convex optimization solvers, however, solving such a problem remains time-consuming when the size of the problem is large. Thus it is important to study how the computational cost of ROTP-type methods might be reduced and how these methods can be accelerated by integrating an acceleration technique such as the heavy-ball (HB) or Nesterov's technique. By using linearization together with a certain binary regularization method,  the so-called nature thresholding (NT) algorithm is developed recently in \cite{ZL22}, whose  computational complexity is significantly lower than that of ROTP${\omega}$ since the NT algorithm is able to avoid solving any optimization problem like (\ref{ROT-model}).  In this paper, we investigate the ROTP-type algorithms from the acceleration perspective by showing that the  HB technique is able to improve the performance of  the ROTP-type algorithms.

The HB method introduced by Polyak \cite{P64} can be seen as a two-step method which combines the momentum term and gradient descent direction. In recent years, HB has found wide applications in image processing, data analysis, distributed optimization and undirected networks \cite{ADR22,GOP17,KBGY22,LRP16,LCWW21,MRJ21,UPS22,XK20}. The theoretical analysis (global convergence and local convergence rate) for HB methods has been investigated by several researchers. For example, the linear convergence rate of HB for unconstrained convex optimization problem was established by Aujol et. al \cite{ADR22}; Mohammadi et. al \cite{MRJ21} analyzed the relation between the convergence rate of HB and its variance amplification when the objective function of the problem is strongly convex and quadratic; Xin and Khan \cite{XK20} showed that the distributed HB method  with appropriate  parameters attains a global $R$-linear rate, and it has potential acceleration compared with some first-order methods for   ill-conditioned problems. Other acceleration techniques including the Nesterov's one can be found in \cite{KBGY22,LRP16,LCWW21,MRJ21}.

In this paper, we merge the optimal $k$-thresholding and HB acceleration technique to form the following algorithms for the SLI problem formulated as (\ref{model}):
\begin{itemize}
\item Heavy-ball-based optimal $k$-thresholding (HBOT),
\item Heavy-ball-based optimal $k$-thresholding pursuit (HBOTP),
\item Heavy-ball-based relaxed optimal $k$-thresholding ($\textrm{HBROT}\omega$),
 \item Heavy-ball-based relaxed optimal $k$-thresholding pursuit ($\textrm{HBROTP}\omega$),
\end{itemize}
where the integer parameter $\omega$ denotes the number of times for solving \eqref{ROT-model} at every iteration.
The global convergence of these algorithms is established in this paper under the restricted isometry property (RIP) introduced by Cand\`{e}s  and Tao \cite{CT05}, and the main results are summarized in Theorems \ref{theorem-OTP} and \ref{theorem-main}.
The performances of HBROTP (i.e., $\textrm{HBROTP}\omega$ with $\omega=1$) and several existing algorithms such as ROTP2 \cite{Z20}, partial gradient ROTP (PGROTP) \cite{MZKS22}, $\ell_1$-minimization   \cite{CDS98},   orthogonal matching pursuit (OMP) \cite{E10,TG07} and projected  linearized Bregman method (PLB) \cite{BPR21}  are compared  through numerical experiments. The phase transition with Gaussian random data is adopted to demonstrate the performances of the proposed algorithms for SLI problems.

The algorithm development for linear inverse problems is usually model-based in the sense that different formulations of the problem require different algorithms. The $\ell_1$-minimization method is naturally applied to the model (\ref{L1-min}) and a more general convex optimization solver can be directly used to handle the LASSO problem \eqref{Lasso}. However, $\ell_1$-minimization and LASSO solvers are not convenient to solve the  problem   (\ref{model}) for which a thresholding method  might be more suitable. The optimal $k$-thresholding method is proposed to enhance the success rates and stability of existing hard thresholding algorithms.   Unlike $\ell_1$-minimization and LASSO  solvers, the hard or optimal $k$-thresholding procedures can reconstruct any prescribed (interested) $k$ significant components of the target signals without the need to reconstruct the whole signal. Also, recent study in \cite{ZL22} indicates that a certain modification of the optimal $k$-thresholding method  may lead to a fast and efficient algorithm  which has  far lower computational cost than most existing algorithms including $\ell_1$-minimization and LASSO solvers. Thus   a further study of the optimal $k$-thresholding algorithms on their acceleration, simplification and modification remains interesting and important from both viewpoints of practical applications and algorithmic development itself.

While our discussion in this paper is focused on hard/optimal thresholding algorithms, it is worth briefly mentioning  the class of soft thresholding methods  which is   widely used for signal processing as well. The soft thresholding method can be derived in different ways. Taking the model \eqref{L1-min} as an example, a soft thresholding method can be developed from the Bregman regularization framework  \cite{YOG08}, which involves solving the convex subproblem \eqref{Lasso} at each iteration. Based on \eqref{Lasso}, using linearization and $\ell_2$-proximity  can lead to the  linearized Bregman (LB) methods \cite{CDB16,YOG08,Y10}, which is a class of soft thresholding methods. Moreover, linearization combined with Krylov subspace projection can also yield a soft thresholding method such as the PLB in \cite{BPR21}. Other soft thresholding methods can be found in \cite{BT09,BSR17,LZC16}. The soft thresholding method needs to select a regularization parameter, but how to select such a parameter so that the algorithm can guarantee to solve a SLI problem  remains an open question. The numerical experiments in Sections \ref{phase_trans} and \ref{deblur_image} indicate that the HBROTP algorithm proposed in this paper might be more stable and robust than $\ell_1$-minimization and PLB for data reconstruction in many cases.

This paper is organized as follows. In Section \ref{sec-prelim}, we introduce some notations, definitions, useful inequalities and   algorithms. In Section \ref{OT-analysis}, we discuss the error bounds and convergence of HBOT and  HBOTP under the  RIP. The error bounds for HBROT$\omega$ and  HBROTP$\omega$ are  given in Section \ref{ROT-analysis}. Numerical results from synthetic signals and real images are reported in Section \ref{simulation}.

\section{Preliminary and algorithms}\label{sec-prelim}
  \subsection{Notations}\label{sec-not}
   Denote by $N:=\{1,2,\ldots,n\}.$  Given a subset  $\Omega\subseteq N,$  $\overline{\Omega}:=N\setminus\Omega$  denotes the complement set  of $\Omega$ and $|\Omega|$ denotes its cardinality. For a vector $z\in\mathbb{R}^{n}$,  the support of $z$ is represented as $ \textrm{supp} (z):=\{i\in N:z_i\neq 0\}$, and the vector $z_\Omega\in \mathbb{R}^n$ is obtained by zeroing out the elements of $z$ supported on $\overline{\Omega}$ and retaining those supported on $\Omega.$  Given the sparse level $k$,   $\mathcal{L}_k(z)$ denotes the index set of the  $k$ largest absolute entries of $z.$ As usual, $\mathcal{H}_k(z):=z_{\mathcal{L}_k(z)}$ is called the hard thresholding of $z.$ The symbols $\|\cdot\|_1$ and $\|\cdot\|_2$ represent $\ell_1$-norm and $\ell_2$-norm of a vector, respectively. Throughout the paper, ${\bf e}$ denotes the vector of ones. $\mathcal{W}^k$ and $ \mathcal{P}^k $ are two sets in $\mathbb{R}^n$ defined as
$$\mathcal{W}^k=\{ w \in \mathbb{R}^n : ~ { \bf e}^T w=k, ~ w \in\{0,1\}^n\}, ~ \mathcal{P}^k=\{ w\in\mathbb{R}^n: { \bf e}^T w=k, ~ 0\leq w \leq \bf e\} . $$

\subsection{Definitions and basic inequalities}\label{sec-inq}
Let us first recall the  restricted isometry property (RIP) of a  matrix   and the optimal $k$-thresholding operator $Z^{\#}_k(\cdot). $

\begin{definition} \cite{CT05} \label{def-RIC}
Given a matrix $A\in \mathbb{R}^{m\times n}$ with $m< n,$ the $k$th order restricted isometry constant (RIC) of $A$, denoted by
$\delta_k,$ is the smallest nonnegative number  $\delta$ such that
\begin{equation}\label{def-RIC-1}
(1-\delta){\left\| u \right\|}^2_2\leq  {\left\| Au \right\|}^2_2\leq (1+\delta){\left\| u \right\|}^2_2
\end{equation}
for all $k$-sparse vectors $u\in \mathbb{R}^n$. The matrix $A$ is
said to satisfy the RIP of order $k$ if $\delta_k<1$.
\end{definition}

\begin{definition} \cite{Z20,ZL21} \label{optimal-operator}
Given a vector $u \in \mathbb{R}^n$, let  $w^* (u)$   be the solution of the binary optimization problem $$\underset{w}{\min}\left\{ \|y- A (u  \circ   w) \|^2_2 : ~ w\in \mathcal{W}^k \right\}.$$ Then the $k$-sparse vector $Z^{\#}_k (u):= u \circ  w^*(u) $ is called the optimal $k$-thresholding of $u , $  and  $Z^{\#}_k(\cdot)$  is called  the optimal $k$-thresholding operator.
\end{definition}

The two lemmas below will be  used for the analysis in Sections \ref{OT-analysis} and \ref{ROT-analysis}.

\begin{lemma} \emph{\cite{F11}} \label{lem-basic-ineq}
Let $u\in \mathbb{R}^n$, $v \in \mathbb{R}^m$, $W\subseteq N$  and $ t \in N .$

{\rm (i)}   If  $ |W\cup  \textrm{supp} (u)|\leq t$, then
$$ \left\|\left[ (I-A^ T A)u\right]_W\right\|_2\leq\delta_t {\| u \|}_2.$$

{\rm (ii)}    If  $ |W|\leq t$, then
$$ \left\| \left( A^ Tv\right)_W \right\|_2\leq\sqrt{1+\delta_t} {\| v \|}_2.$$

\end{lemma}

\begin{lemma} \emph{\cite{SZZM22}}\label{lem-two-level-geometric}
Let $\{a^p\}\subseteq\mathbb{R} ~ (p= 0, 1, \dots)$ be a nonnegative sequence satisfying
\begin{equation*}\label{lem-tlgeom-1}
a^{p+1}\leq b_1a^{p}+b_2a^{p-1}+b_3
\end{equation*}
 for   $  p\geq 1,$ where  $b_1,b_2$ and $ b_3\geq 0 $ are constants and $b_1+b_2<1$. Then
\begin{equation*}\label{lem-tlgeom-2}
a^{p}\leq\theta^{p-1}\left(a^{1}+(\theta-b_1)a^{0}\right)+\frac{b_3}{1-\theta}
\end{equation*}
 for    $  p\geq 2, $ where $0\leq \theta <1$ is a constant given by $ \theta=(b_1+\sqrt{b_1^2+4b_2})/2 <1. $
\end{lemma}

\subsection{Algorithms}\label{sec-alg}

Given iterates $ x^{p-1}$ and $ x^p $, the heavy-ball search direction is defined as $$ d^{p}=\alpha A^ T(y-Ax^p)+\beta (x^p-x^{p-1}),$$ where $ \alpha> 0$ and $  \beta\geq 0  $ are two parameters. We use
 the optimal $k$-thresholding operator  $Z^{\#}_k(\cdot)$ to generate the new iterate $x^{p+1}$, i.e., $$x^{p+1}=Z^{\#}_k(x^p+d^p),$$ which is called the heavy-ball-based optimal $k$-thresholding (HBOT) algorithm.   Combining HBOT and orthogonal projection (i.e., the least squares problem (\ref{algorithm-OTP-3}) below) leads to
 the heavy-ball-based optimal $k$-thresholding pursuit (HBOTP) algorithm.
   HBOT and HBOTP can be seen  as the multi-step extensions of the OT and OTP algorithms in \cite{Z20,ZL21}.
The two algorithms are formally described as follows.

\vskip 0.08in
 \textbf{$\textrm{HBOT}$  and  $\textrm{HBOTP}$ algorithms. }
Input the data $(A,y,k)$ and two initial points $x^0 $ and $ x^1$. Choose the parameters $\alpha>0$ and $ \beta\geq 0$.
 \begin{itemize}
 \item [S1] At $x^p$, set
 \begin{equation}\label{algorithm-OTP-1}
u^p=x^p-\alpha A^ T(Ax^p-y)+\beta (x^p-x^{p-1}).
\end{equation}

 \item [S2]  Solve the  optimization problem
  \begin{equation}\label{algorithm-OTP-2}
w^*=\arg \min_{w} \{ {\| y-A(u^p\circ{w})\|}_2^2: ~ { \bf e}^T w=k, ~ w \in\{0,1\}^n\}.
\end{equation}

 \item[S3]   Generate the next iterate $x^{p+1}  $ as follows:
 \begin{itemize} \item[] For HBOT, let $x^{p+1}=u^p\circ{w}^*$.
 \item[] For HBOTP, let  $S^{p+1}= \textrm{supp} (u^p\circ{w}^*)$ and  $x^{p+1}$ be the solution to the least squares problem
  \begin{equation}\label{algorithm-OTP-3}
x^{p+1}=\arg \min_{x\in \mathbb{R}^n}  \{ {\| y-Ax\|}_2^2:  \textrm{supp} (x)\subseteq S^{p+1}\}.
\end{equation}
\end{itemize}
Repeat S1-S3 above until a certain stopping criterion is met.
\end{itemize}

In general, the computational cost for solving the binary optimization problem \eqref{algorithm-OTP-2} is usually high  \cite{BT18,CAP08}. Replacing \eqref{algorithm-OTP-2} by its convex relaxation $$\arg \min_{w} \left\{ \|y- A (u  \circ   w) \|^2_2 : ~ w\in \mathcal{P}^k \right\} $$  yields the next heavy-ball-based relaxed optimal $k$-thresholding (HBROT$\omega$)   and the heavy-ball-based relaxed optimal $k$-thresholding pursuit (HBROTP$\omega$) algorithms, where $\omega$ represents the times of solving such a convex relaxation problem at each iteration (which, as pointed out in \cite{Z20}, can be interpreted as the times of data compression within each iteration).  As $\omega=1$, we simply use HBROT and HBROTP to denote the algorithms HBROT$1$ and HBROTP$1$, respectively. Clearly, when $\alpha=1$ and $\beta=0$, HBROT$\omega$ and HBROTP$\omega$ reduce, respectively,  to  ROT$\omega$ and ROTP$\omega$ in \cite{ZL21}.

\vskip 0.08in
 \textbf{$\textrm{HBROT}\omega$ and $\textrm{HBROTP}\omega$ algorithms.}
Input the data $(A,y,k)$, two initial points $x^0, x^1$ and $\omega$. Choose the parameters $\alpha>0$ and $ \beta\geq 0$.
\begin{itemize}
 \item [S1] At $x^p$, calculate $u^p$ by \eqref{algorithm-OTP-1}.

 \item [S2]  Set $v\leftarrow u^p$. Perform the following loops to produce
the vectors ${w}^{(j)} (j=1,\ldots,\omega)$:

  \textbf{for}  $j=1:\omega$   \textbf{do}
\begin{equation}\label{algorithm-ROTP-2}
{w}^{(j)}=\arg \min_{w}  \{ {\| y-A(v\circ{w})\|}_2^2: \ \ { \bf e}^T w=k, \ \ 0\leq w \leq \bf e\},
\end{equation}
and set $v\leftarrow v\circ{w}^{(j)}$.

  \textbf{end}

 \item [S3]  Let $x^{\sharp}=\mathcal{H}_k(u^p\circ{w}^{(1)}\circ\cdots\circ {w}^{(\omega)})$. Generate the next iterate $x^{p+1}$ as follows:
 \begin{itemize}
 \item[] For $\textrm{HBROT}\omega$, let $x^{p+1}=x^{\sharp}$.
\item[] For $\textrm{HBROTP}\omega$, let $S^{p+1}= \textrm{supp} (x^{\sharp})$, and  $x^{p+1}$ be the solution to the least squares problem
 \begin{equation}\label{algorithm-ROTP-3}
x^{p+1}=\arg \min_{x\in \mathbb{R}^n}  \{ {\| y-Ax\|}_2^2:  \textrm{supp} (x)\subseteq S^{p+1}\}.
\end{equation}
\end{itemize}
 Repeat S1-S3 above until a certain stopping criterion is met.
\end{itemize}

The choice of stopping criterions depends on the application scenarios. For instance, one can simply prescribe the maximum number of iterations, $p_{\max},$ which allows the algorithm to perform a total of $ p_{\max} $ iterations. One can also terminate the algorithm when  $ \|y-Ax^p\|_2\leq \varepsilon,$ where $ \varepsilon>0$ is a prescribed tolerance depending on the noise level.

\begin{remark}\label{sparsity-complexity}
 The common feature of the proposed algorithms and existing hard-type thresholding algorithms is that  the solutions generated by the algorithms  depend  on the input value of $k$, which reflects the user's interest in reconstructing how many significant components of the target signal $x^*$ whose sparsity level is denoted by $k^*.$ In many scenarios, one needs to reconstruct only a few largest components in magnitude of the target signal, instead of the whole signal. In such cases, the user is free to set the desired number $k$ for the proposed algorithms. The quality of reconstruction depends on the input value of $k.$  In fact, the main theorems established in later sections imply that under the RIP of certain order $\widehat{k},$ the solution generated by the algorithms is the best $k$-term approximation to the true signal $x^*$  when $ k < k^*,$ and it  coincides with $x^*$ when $k$ satisfies that $k^* \leq k < \widehat{k}.$ When  $k \geq \widehat{k},$ there would be no guarantee for the proposed algorithms (including existing ones) to recover the signal.
 If the user expects to reconstruct the whole signal as possible, some information from theory and numerical experiments might be useful for the choice of $k.$ For instance, we may choose $k$ as follows.
 \begin{itemize}

 \item[1)] The prior information on the sparsity level $k^*$ of the signal might be available in some situations. In this case, just set $ k= k^*. $

 \item[2)] It is well known that the signal can be very likely to be recovered by a certain algorithm if its sparsity level $ k^*$  is lower than the half of the spark of the measurement matrix in $ \mathbb{R}^{m\times n} $ \cite{E10}.  So it makes sense to choose  $k < (m+1)/2$ since the spark is bounded above by $m+1.$

 \item[3)] A large body of simulations and  applications indicate that many algorithms work well when the sparsity level of signal is lower than $m/3,$ and many such signals can be generally reconstructed by some existing algorithms. Thus it is also reasonable to set   $k \leq m/3 $ in the proposed algorithms in order to achieve a better chance for the signal to be recovered.

 \item[4)] The number $k$ can  be also suggested by experiments including the phase transition of algorithms which sheds light on certain relation  between the factors $(k,m,n)$ and the recovery success of signals by given algorithms.
         \end{itemize}
\end{remark}

\begin{remark}
  Since $x^{p-1}, x^p$ are two $k$-sparse vectors in $\mathbb{R}^{ n}$ and $A\in \mathbb{R}^{m\times n}$, the computations of $Ax^p$ and $\beta (x^p-x^{p-1})$  in \eqref{algorithm-OTP-1} need at most $m k$ and $2k$ multiplication operations, respectively. Thus  S1 in HBROTP$\omega$ requires at most $mn+m+mk+2k$ multiplication operations.
As pointed out in \cite[Section 5.1]{ZL21},  S2 and S3 in HBROTP$\omega$ requires $O(n^{3.5}L+n\log k)+(mk^2+k^3/3)$ flops, in which $L$ is the length of the problem data encoding in binary. Since $k\ll m$, the computational complexity of HBROTP$\omega$ at each iteration is about $O(n^{3.5}L+mn+n\log k+mk^2)$.

\end{remark}

\section{Analysis of  HBOT and  HBOTP }\label{OT-analysis}

In this section, we establish the error bounds for HBOT and  HBOTP under the RIP of order $k$ or $k+1$. Taking into account the noise influence, the error bound provides the estimation of the distance between the problem solution and the iterates generated by the algorithms. Thus the error bound is an important measurement of the quality of iterates as the approximation to the true solution of the linear inverse problem. In noiseless situations, the error bound implies the global convergence of the algorithms under the RIP assumption.   Let us first introduce
the following property, which is a combination of Lemmas 3.3 and 3.6  in \cite{ZL21}.

\begin{lemma} \emph{\cite{ZL21}} \label{lem-Az-sk}
Let $z$ be a $(2k)$-sparse vector. Then
$ \|Az\|_2^2\geq(1-2\delta_k-\delta_{k+s(k)})\|z\|_2^2$
where
\begin{align}\label{lem-Az-sk-2}
 s(k)=\left\{
\begin{array}{ll}
1, ~ &\textrm{ if }~ k~ \textrm{is an odd  number},\\
0, ~&\textrm{ if }~ k~ \textrm{is an even number}.\\
\end{array}
\right.
\end{align}
\end{lemma}

Note that Lemma 3.4 in \cite{ZL21} was established for the sparsity level $k$ being an even number. We now establish the similar result even when $k$ is odd.
\begin{lemma}  \label{lem-IATA-sqrt5}
Let $h,z\in \mathbb{R}^n$ be two $k$-sparse vectors, and let $\widehat{w} \in \mathcal{W}^k$ be any $k$-sparse binary vector satisfied $ \textrm{supp} (h)\subseteq  \textrm{supp} (\widehat{w})$, then
\begin{equation}\label{lem-IATA-sqrt5-1}
  \| [ (I-A^ T A)(h-z)]\circ\widehat{w} \|_2\leq \sqrt{5} \delta_{k+s(k)} {\| h-z\|}_2,
\end{equation}
where $s(k)$ is given by \eqref{lem-Az-sk-2}.
\end{lemma}

\begin{proof}
For given vectors $h,z,\widehat{w}$ satisfying the conditions of the lemma, from \cite[Lemma 3.4]{ZL21}, we get
\begin{equation}\label{pf-sqrt5-del-1}
 \| [ (I-A^ T A)(h-z)]\circ\widehat{w} \|_2\leq \sqrt{5} \delta_{k} {\| h-z\|}_2
\end{equation}
  for even number $k$. Therefore, we just need to show that \eqref{lem-IATA-sqrt5-1} also holds when $k$ is an odd number.

Indeed, assume that $ k$ is an odd number. Taking  a $(k+1)$-sparse binary vector $\overline{w} \in\mathcal{W}^{k+1}$ such that  ${\rm supp} (\widehat{w})\subseteq{\rm supp} (\overline{w})$,  we obtain
  \begin{align}\label{pf-sqrt5-del-2}
 \| [ (I-A^ T A)(h-z)]\circ\widehat{w} \|_2
 &= \left\| \left [ (I-A^ T A)(h-z)\right ]_{\textrm{supp} (\widehat{w})}\right\|_2\nonumber \\
 &\leq   \left\| \left[ (I-A^ T A)(h-z) \right ]_{\textrm{supp} (\overline{w})}\right\|_2 \nonumber \\
 &= \| [ (I-A^ T A)(h-z)]\circ\overline{w} \|_2.
\end{align}
  As $h$ and $z$ are two $k$-sparse vectors,  they are also  $(k+1)$-sparse vectors. Since $ \textrm{supp} (h)\subseteq  \textrm{supp} (\widehat{w})\subseteq  \textrm{supp} (\overline{w})$, and since $k+1$ is even (when $k$ is odd), applying  \eqref{pf-sqrt5-del-1} to this case yields
  \begin{equation}\label{pf-sqrt5-del-3}
  \| [ (I-A^ T A)(h-z)]\circ\overline{w} \|_2\leq \sqrt{5} \delta_{k+1} {\| h-z\|}_2.
\end{equation}
Combining \eqref{pf-sqrt5-del-2} and \eqref{pf-sqrt5-del-3}, we obtain
\begin{equation*}\label{pf-sqrt5-del-4}
  \| [ (I-A^ T A)(h-z)]\circ\widehat{w} \|_2\leq \sqrt{5} \delta_{k+1} {\| h-z\|}_2
\end{equation*}
 for odd number $k$. We  conclude that \eqref{lem-IATA-sqrt5-1} holds for any positive integer  $k$.  \hfill $\Box$
\end{proof}

The main results for HBOT and HBOTP are summarized as follows.

\begin{theorem}\label{theorem-OTP}
Let $x \in \mathbb{R}^n$ be a solution to the system $y= Ax+\nu$ where $\nu$ is a noise vector. Assume that the RIC, $\delta_{k+s(k)},$ of $A$ and the parameters $\alpha,\beta$ in HBOT and HBOTP satisfy that  $ \delta_{k+s(k)}<\gamma^*$ and
\begin{equation}\label{th-OTP-0}
 0\leq\beta<\frac{1+1/\eta}{1+\sqrt{5} \delta_{k+s(k)}}-1,~
\frac{1+2\beta-1/\eta}{1-\sqrt{5} \delta_{k+s(k)}}<\alpha<\frac{1+1/\eta}{1+\sqrt{5} \delta_{k+s(k)}},
\end{equation}
where $\gamma^*(\approx 0.2274)$ is the unique root of the equation $5\gamma^3+5\gamma^2+3\gamma-1=0$ in the interval $(0,1)$,  $s(k)$ is given by \eqref{lem-Az-sk-2} and $\eta:=\sqrt{\frac{1+\delta_k}{1-2\delta_k-\delta_{k+s(k)}}}$.
Then the sequence $\{ x^p\}$ generated by  HBOT or HBOTP obeys
\begin{equation}\label{th-OTP-1}
 {\| x_S-x^p\|}_2\leq C_1\theta^{p-1} +
C_2\|  \nu'\|_2,
\end{equation}
where $S:=\mathcal{L}_k(x)$, $\nu':=\nu+Ax_{\overline{S}}$, and the quantities  $C_1,C_2$  are defined as
\begin{equation}\label{th-OTP-2}
C_1={\|  x_S-x^1\|}_2+({\theta}-{b}){\|  x_S-x^0\|}_2, ~
C_2=\frac{2+(1+\delta_k)\alpha}{(1-\theta) \sqrt{1-2\delta_k-\delta_{k+s(k)}}},
\end{equation}
and $\theta:=(b+\sqrt{b^2+4\eta\beta})/2<1$ is ensured under the conditions \eqref{th-OTP-0} and  the constant $b$ is given by
\begin{equation}\label{th-OTP-3}
b:=\eta\left(|1+\beta-\alpha|+\sqrt{5}\alpha \delta_{k+s(k)}\right).
\end{equation}
\end{theorem}

\begin{proof}
From \eqref{algorithm-OTP-1}, we have
 \begin{equation}\label{th-OTP-pf-1}
u^p-x_S=(1-\alpha+\beta)(x^p-x_S)+\alpha (I-A^ TA)(x^p-x_S)-\beta (x^{p-1}-x_S)
+\alpha A^T \nu',
\end{equation}
where $S=\mathcal{L}_k(x)$ and $\nu'=\nu+Ax_{\overline{S}}$. Let  $\widehat{w}\in \mathcal{W}^k$ be a $k$-sparse binary vector such that $ \textrm{supp} (x_S) \subseteq  \textrm{supp} (\widehat{w})$. Then $x_S=x_S\circ \widehat{w}$. Since $(x_S-u^p )\circ \widehat{w}$ is a $k$-sparse vector and $y=Ax_S+\nu'$, we have
\begin{align}\label{th-OTP-pf-2}
{\|y-A(u^p\circ\widehat{w})\|}_2=&{\|A(x_S-u^p\circ\widehat{w})+\nu'\|}_2\nonumber\\
\leq &{\|A[(x_S-u^p)\circ\widehat{w}]\|_2+\|\nu'\|}_2 \nonumber \\
\leq & \sqrt{1+\delta_k} {\|(x_S-u^p)\circ\widehat{w}\|_2+\| \nu'\|}_2,
\end{align}
where the last inequality is obtained by using \eqref{def-RIC-1}. From \eqref{th-OTP-pf-1}, one has
\begin{align}\label{th-OTP-pf-3}
 & \|(x_S      -u^p)\circ\widehat{w}\|_2  \nonumber \\
 & \leq     |1-\alpha+\beta|\cdot\|(x^p-x_S)\circ\widehat{w}\|_2
+\alpha \| [(I-A^ TA)(x^p-x_S)]\circ\widehat{w}\|_2\nonumber \\
& ~~~ +\beta\| (x^{p-1}-x_S)\circ\widehat{w}\|_2+\alpha\| ( A^T \nu')\circ\widehat{w}\|_2.
\end{align}
Since $x_S,x^p,\widehat{w}$ are  $k$-sparse vectors and $ \textrm{supp} (x_S) \subseteq  \textrm{supp} (\widehat{w})$, by using Lemmas \ref{lem-IATA-sqrt5} and \ref{lem-basic-ineq} (ii), we obtain
\begin{equation}\label{th-OTP-pf-4}
\| [(I-A^ TA)(x^p-x_S)]\circ\widehat{w}\|_2\leq \sqrt{5} \delta_{k+s(k)}\|x^p-x_S\|_2
\end{equation}
and
\begin{equation}\label{th-OTP-pf-5}
\| ( A^T \nu')\circ\widehat{w}\|_2= \left\| ( A^T \nu')_{ \textrm{supp} (\widehat{w})}\right\|_2\leq \sqrt{1+\delta_k}\| \nu'\|_2.
\end{equation}
Substituting \eqref{th-OTP-pf-4} and  \eqref{th-OTP-pf-5} into \eqref{th-OTP-pf-3}  yields
\begin{align*}
 \|(x_S-u^p)\circ\widehat{w}\|_2
  & \leq  |1-\alpha+\beta|\cdot\|x^p-x_S\|_2
+\alpha\sqrt{5} \delta_{k+s(k)}\|x^p-x_S\|_2  \nonumber \\
 & ~~~ +\beta\| x^{p-1}-x_S\|_2+ \alpha\sqrt{1+\delta_k}\| \nu'\|_2\nonumber \\
& =    (|1+\beta-\alpha|+\sqrt{5}\alpha \delta_{k+s(k)})\|x^p-x_S\|_2+
\beta\| x^{p-1}-x_S\|_2 \nonumber \\
 & ~~~ + \alpha\sqrt{1+\delta_k}\| \nu'\|_2.
\end{align*}
It follows from \eqref{th-OTP-pf-2} that
\begin{align}\label{th-OTP-pf-7}
{\|y-A(u^p\circ\widehat{w})\|}_2
\leq& \sqrt{1+\delta_k} \left(|1+\beta-\alpha|+\sqrt{5}\alpha \delta_{k+s(k)}\right)\|x^p-x_S\|_2\nonumber\\
&+\beta\sqrt{1+\delta_k}\| x^{p-1}-x_S\|_2+ [1+{(1+\delta_k)}\alpha]\| \nu'\|_2.
\end{align}
Since $x^{p+1}=u^p\circ{w}^*$ in HBOT or $x^{p+1}$ is the optimal solution of \eqref{algorithm-OTP-3} in HBOTP,   the sequence $\{x^{p}\}$ generated by HBOT or HBOTP satisfies
\begin{equation}\label{th-OTP-pf-8}
{\|y-Ax^{p+1}\|}_2\leq{\|y-A(u^p\circ {w^*})\|}_2\leq{\|y-A(u^p\circ {w})\|}_2
\end{equation}
 for all $ w\in \mathcal{W}^k,$ where the second inequality follows from  \eqref{algorithm-OTP-2}.
For $\widehat{w}\in \mathcal{W}^k$, it follows  from \eqref{th-OTP-pf-8} that
\begin{equation}\label{th-OTP-pf-9}
{\|y-Ax^{p+1}\|}_2\leq{\|y-A(u^p\circ \widehat{w})\|}_2.
\end{equation}
As  $x_S-x^{p+1}$ is a $(2k)$-sparse vector, by using  Lemma \ref{lem-Az-sk}, one has
\begin{align}\label{th-OTP-pf-10}
{\|y-Ax^{p+1}\|}_2=&{\|A(x_S-x^{p+1})+\nu'\|}_2\nonumber\\
\geq &{\|A(x_S-x^{p+1})\|_2-\|\nu'\|}_2\nonumber\\
\geq&\sqrt{1-2\delta_k-\delta_{k+s(k)}}{\|x_S-x^{p+1}\|_2-\|\nu'\|}_2.
\end{align}
Combining \eqref{th-OTP-pf-7}, \eqref{th-OTP-pf-9} and \eqref{th-OTP-pf-10} yields
\begin{align}\label{th-OTP-pf-11}
 \| x^{p+1}-x_S\|_2
& \leq  \eta (|1+\beta-\alpha|+\sqrt{5}\alpha \delta_{k+s(k)})\|x^p-x_S\|_2
+\eta\beta\| x^{p-1}-x_S\|_2   \nonumber\\
 & ~~~   +\frac{2+{(1+\delta_k)}\alpha}{\sqrt{1-2\delta_k-\delta_{k+s(k)}}} \| \nu'\|_2\nonumber\\
& =  b\|x^p-x_S\|_2+
\eta\beta\| x^{p-1}-x_S\|_2+(1-\theta)C_2\| \nu'\|_2,
\end{align}
where $\eta,b,\theta,C_2$  are given exactly as in Theorem \ref{theorem-OTP}. Since $\delta_k\leq\delta_{k+s(k)}<\gamma^*$, we have
\begin{equation*}\label{th-OTP-pf-12}
\eta\sqrt{5} \delta_{k+s(k)}=\sqrt{5} \delta_{k+s(k)}\sqrt{\frac{1+\delta_k}{1-2\delta_k-\delta_{k+s(k)}}}<\sqrt{5} \gamma^*\sqrt{\frac{1+\gamma^*}{1-3\gamma^*}}=1,
\end{equation*}
where the last equality follows from the fact that $\gamma^*$ is the  root of $5\gamma^3+5\gamma^2+3\gamma=1$ in $(0,1)$.
It implies that $0<\frac{1+1/\eta}{1+\sqrt{5} \delta_{k+s(k)}}-1$, which shows that the range of  $\beta$  in \eqref{th-OTP-0} is well defined. Furthermore,  the first inequality  in \eqref{th-OTP-0}  implies that
\[
\frac{1+2\beta-1/\eta}{1-\sqrt{5} \delta_{k+s(k)}}< 1+\beta<\frac{1+1/\eta}{1+\sqrt{5} \delta_{k+s(k)}},\]
which indicates that  the range for $\alpha$  in \eqref{th-OTP-0} is also well defined.
Combining \eqref{th-OTP-3} with \eqref{th-OTP-0}, we deduce that
\begin{align*}\label{th-OTP-pf-13}
b=&\eta\left(|1+\beta-\alpha|+\sqrt{5}\alpha \delta_{k+s(k)}\right)\nonumber\\
=&\left\{
\begin{array}{ll}
\eta\left[1+\beta-\alpha(1-\sqrt{5} \delta_{k+s(k)})\right],\ \ &\textrm{ if } \frac{1+2\beta-1/\eta}{1-\sqrt{5} \delta_{k+s(k)}}<\alpha\leq 1+\beta, \\
\eta\left[-1-\beta+\alpha(1+\sqrt{5} \delta_{k+s(k)})\right],\ \ &\textrm{ if } 1+\beta< \alpha< \frac{1+1/\eta}{1+\sqrt{5} \delta_{k+s(k)}},
\end{array}
\right.\\
<& 1-\eta\beta,\nonumber
\end{align*}
which means that the relation  \eqref{th-OTP-pf-11} obeys the conditions of Lemma \ref{lem-two-level-geometric}. It  follows  from
Lemma \ref{lem-two-level-geometric} that
\eqref{th-OTP-1}  holds with $\theta=\frac{b+\sqrt{b^2+4\eta\beta}}{2}<1$ and  $C_1,C_2$ are given by \eqref{th-OTP-2}. \hfill $\Box$
\end{proof}

The error bound (\ref{th-OTP-1}) indicates that the iterate $ x^p$ generated by the algorithms can approximate $x_S,$ the significant components of the solution to the linear inverse problem. In particular, we immediately obtain the following convergence result for the algorithms.

\begin{corollary}
Let $x \in \mathbb{R}^n$ be a $k$-sparse solution to the system $y = Ax .$   Assume that the RIC, $\delta_{k+s(k)},$ of $A$ and the parameters $\alpha,\beta$ in HBOT and HBOTP satisfy the conditions of  Theorem \ref{theorem-OTP}. Then the sequence $\{ x^p\}$ generated by  HBOT or HBOTP obeys that
$
 {\|x-x^p\|}_2\leq C_1\theta^{p-1},
$
where the constant $C_1$ is defined in Theorem \ref{theorem-OTP}. Thus the sequence $\{ x^p\}$ generated by  HBOT or HBOTP converges to $x.$
\end{corollary}

\section{Analysis of HBROT$\omega$ and  HBROTP$\omega$}\label{ROT-analysis}
In this section, we establish the error bounds for HBROT$\omega$ and  HBROTP$\omega.$ The analysis is far from being trivial. We need a few useful technical results before we actually establish the error bounds.
We first recall a helpful lemma concerning the polytope  $\mathcal{P}^k$, which is the special case of Lemma 4.2 with $\tau=k$ in \cite{ZL21}.

\begin{lemma}\emph{\cite{ZL21}} \label{lemma4.2}
Given an index set $\Lambda \subseteq N$ and a vector $w\in\mathcal{P}^k, $ decompose  $w_\Lambda$ as the sum of
$k$-sparse vectors:
$w_\Lambda = \sum_{j=1}^q w_{\Lambda_j} $,
where $q:=\lceil \frac{|\Lambda|}{k}\rceil$, $\Lambda = \bigcup_{j=1}^q {\Lambda_j} $ and $\Lambda_1:=\mathcal{L}_k(\omega_\Lambda)$, $\Lambda_2:=\mathcal{L}_k(\omega_{\Lambda\setminus\Lambda_1})$  and so on.
Then $$\sum_{j=1}^q\|w_{\Lambda_j} \|_\infty < 2.$$
\end{lemma}

We now give an inequality concerning the norms $\|\cdot\|_2,~ \|\cdot\|_1$ and $\|\cdot\|_\infty.$ This inequality  is a  modification of Lemma 6.14 in \cite{FR13}, but tailored to the need of the later analysis in this paper.

\begin{lemma}\label{lemma-norm-L2-L1}
Let $h\in \mathbb{R}^r\setminus\{0\}$ be a vector with  $r\geq 2$, and let $\zeta_1> \zeta_2$ be two positive numbers  such that $\|h\|_1\leq \zeta_1$ and $\|h\|_{\infty}\leq \zeta_2$. Then
\begin{equation}\label{lem-norm-L2-L1-1}
\|h\|_2\leq\left\{
\begin{array}{cl}
g(r),&~\textrm{ if } ~ r\leq t_0,\\
 \min\{g(t_0),g(t_0+1)\},& ~\textrm{ if } ~  r\geq t_0+1,\\
\end{array}\right.
\end{equation}
where $t_0:=\lfloor\frac{4\zeta_1}{\zeta_2}\rfloor$ and
\begin{equation}\label{lem-norm-L2-L1-pf-0}
g(j):=\frac{1}{\sqrt{j}}\zeta_1+\frac{\sqrt{j}}{4}\zeta_2,~~j\in (0,+\infty),
\end{equation}
 is strictly  decreasing in the interval  $(0,\frac{4\zeta_1}{\zeta_2}]$ and strictly increasing in the interval $[\frac{4\zeta_1}{\zeta_2},+\infty)$.
\end{lemma}

\begin{proof}
Without loss of generality, we  assume that $h$ is a nonnegative vector. Sort the components of $h$ into descending order, and denote such ordered components by $z_1\geq z_2\geq \cdots \geq z_r\geq 0$ and $z=(z_1,\ldots,z_r)^T$. Thus, $\|z\|_q=\|h\|_q$ for  $q\geq 1.$  For a given positive integer $s$ and $a_1\geq a_2\geq \cdots \geq a_s\geq 0$, from \cite[Lemma 6.14]{FR13}, one has
\begin{equation}\label{lem-norm-L2-L1-pf-1}
\sqrt{a_1^2+\cdots+a_s^2}\leq \frac{a_1+\cdots+a_s}{\sqrt{s}}+\frac{\sqrt{s}}{4}(a_1-a_s).
\end{equation}

There are only two cases according to the relation between $r$ and $t_0$.

{\bf Case 1}.   $r\leq t_0.$ By using \eqref{lem-norm-L2-L1-pf-0} and \eqref{lem-norm-L2-L1-pf-1}, we have
\begin{equation}\label{lem-norm-L2-L1-pf-2}
\|z\|_2\leq \frac{\|z\|_1}{\sqrt{r}}+\frac{\sqrt{r}}{4}(z_1-z_r)\leq \frac{\|z\|_1}{\sqrt{r}}+\frac{\sqrt{r}}{4}\|z\|_\infty\leq  \frac{1}{\sqrt{r}}\zeta_1+\frac{\sqrt{r}}{4}\zeta_2   =g(r).
\end{equation}

{\bf Case 2}.   $r\geq t_0+1. $ Denote by  $ t:=\arg \min_{j}\{g(j):j=t_0,t_0+1\}$ and let  $r_1,r_2$ be nonnegative integers such that  $r=r_1t+r_2(0\leq r_2<t)$.
Decompose  $z$ as the sum of
$t$-sparse vectors:
$z = \sum_{j=1}^{r_1+1} z_{Q_j} $,
where $Q_j:=\{(j-1)t+1,\ldots,jt\}$ with $ j=1,\ldots,r_1, $ and  $Q_{r_1+1}:=\{r_1t+1,\ldots,r_1t+r_2\}$.

Firstly, we consider the case $r_2>0$. With the aid of  \eqref{lem-norm-L2-L1-pf-1}, we see that
\begin{equation}\label{lem-norm-L2-L1-pf-3}
\|z_{Q_j}\|_2\leq \frac{\|z_{Q_j}\|_1}{\sqrt{t}}+\frac{\sqrt{t}}{4}\left(z_{(j-1)t+1}-z_{jt}\right),\ \ j=1,\ldots,r_1,
\end{equation}
and
\begin{equation}\label{lem-norm-L2-L1-pf-4}
\|z_{Q_{r_1+1}}\|_2\leq \frac{\|z_{Q_{r_1+1}}\|_1}{\sqrt{t}}+\frac{\sqrt{t}}{4}z_{r_1t+1},
\end{equation}
which is ensured under the conditions $a_1=z_{r_1t+1},\ldots, a_{r_2}=z_{r_1t+r_2}$ and $a_{r_2+1}=\ldots =a_{t}=0$. Merging \eqref{lem-norm-L2-L1-pf-3}  with \eqref{lem-norm-L2-L1-pf-4}, one has
\begin{equation*}\label{lem-norm-L2-L1-pf-5}
\|z\|_2=\left\| \sum_{j=1}^{r_1+1} z_{Q_j} \right\|_2\leq\sum_{j=1}^{r_1+1}\|z_{Q_j}\|_2\leq \frac{1}{\sqrt{t}}\sum_{j=1}^{r_1+1}\|z_{Q_j}\|_1+\frac{\sqrt{t}}{4}\mu
\end{equation*}
with
\begin{equation*}\label{lem-norm-L2-L1-pf-6}
\mu:=\sum_{j=1}^{r_1}\left(z_{(j-1)t+1}-z_{jt}\right)+z_{r_1t+1}
=z_1-\sum_{j=1}^{r_1}\left(z_{jt}-z_{jt+1}\right)\leq z_1,
\end{equation*}
where the inequality is resulted from $z_1\geq z_2\geq \cdots \geq z_r$. It follows that
\begin{equation}\label{lem-norm-L2-L1-pf-7}
\|z\|_2\leq \frac{1}{\sqrt{t}}\|z\|_1+\frac{\sqrt{t}}{4}z_1\leq \frac{1}{\sqrt{t}}\zeta_1+\frac{\sqrt{t}}{4}\zeta_2=g(t),
\end{equation}
where the second inequality is ensured by $\|z\|_1=\|h\|_1\leq \zeta_1$ and $z_1=\|h\|_{\infty}\leq \zeta_2$.

Secondly, we now consider the case  $r_2=0$, which means $Q_{r_1+1}=\emptyset$ and $z = \sum_{j=1}^{r_1} z_{Q_j}$. Hence, by using \eqref{lem-norm-L2-L1-pf-1}, we obtain
\begin{equation}\label{lem-norm-L2-L1-pf-8-0}
\|z\|_2 \leq\sum_{j=1}^{r_1}\|z_{Q_j}\|_2\leq \frac{1}{\sqrt{t}}\sum_{j=1}^{r_1}\|z_{Q_j}\|_1+\frac{\sqrt{t}}{4}\sum_{j=1}^{r_1}\left(z_{(j-1)t+1}-z_{jt}\right).
\end{equation}
For $z_1\geq z_2\geq \cdots \geq z_r\geq 0$, we have
 \begin{equation}\label{lem-norm-L2-L1-pf-8-1}
\sum_{j=1}^{r_1}\left(z_{(j-1)t+1}-z_{jt}\right)=z_1-\sum_{j=1}^{r_1-1}\left(z_{jt}-z_{jt+1}\right)-z_{r_1}\leq z_1=\|z\|_\infty.
\end{equation}
Merging \eqref{lem-norm-L2-L1-pf-8-0} with \eqref{lem-norm-L2-L1-pf-8-1} leads to
\begin{equation}\label{lem-norm-L2-L1-pf-8}
\|z\|_2 \leq \frac{1}{\sqrt{t}}\|z\|_1+\frac{\sqrt{t}}{4}\|z\|_\infty\leq \frac{1}{\sqrt{t}}\zeta_1+\frac{\sqrt{t}}{4}\zeta_2=g(t).
\end{equation}
Combining \eqref{lem-norm-L2-L1-pf-2},   \eqref{lem-norm-L2-L1-pf-7}, \eqref{lem-norm-L2-L1-pf-8} with  $\|z\|_q=\|h\|_q (q\geq 1)$, we  obtain the relation \eqref{lem-norm-L2-L1-1} directly. \hfill $\Box$
\end{proof}

Now, we use an example to show that the upper bound of $\|h\|_2$ in Lemma \ref{lemma-norm-L2-L1}  is tighter than that of Lemma 6.14 in \cite{FR13} in some situations.

\begin{example}
Let $h=(1,\epsilon_1,\ldots,\epsilon_{14},\epsilon_0)^T\in \mathbb{R}^{16}$, where $\epsilon_j\geq \epsilon_0 ~ (j=1,\ldots,14), $ $ \sum_{j=1}^{14}\epsilon_j=1-\epsilon_0$ and $ \epsilon_0\in(0,1/15]$. Hence, $\|h\|_1=2$ and $ \|h\|_\infty=1$. Set $\zeta_1=2$ and $\zeta_2=1.$ Then $t_0=\frac{4\zeta_1}{\zeta_2}=8$. The upper bound of $\|h\|_2$ can be given by  $\|h\|_2\leq 1.5-\epsilon_0$ in \eqref{lem-norm-L2-L1-pf-1} and $\|h\|_2\leq g(8)=\sqrt{2}$ in \eqref{lem-norm-L2-L1-1}, respectively. Since $1.5-\epsilon_0>\sqrt{2}$ for $ \epsilon_0\in(0,1/15]$,  we see that the  upper bound of $\|h\|_2$ given by \eqref{lem-norm-L2-L1-1} is tighter than that of  \eqref{lem-norm-L2-L1-pf-1} if $r>t_0+1$ and $\epsilon_0=\min_{1\leq i\leq  r} |h_i|$ is small enough.
\end{example}

Taking $h=( \|w_{\Lambda_1} \|_\infty,\ldots, \|w_{\Lambda_q} \|_\infty)^T$ with $q=\lceil \frac{|\Lambda|}{k}\rceil$, by using Lemma \ref{lemma4.2}, we have $\|h\|_1<2$ and $ \|h\|_\infty= \|w_{\Lambda_1}\|_\infty\leq 1$. Hence, by setting $\zeta_1=2$ and $\zeta_2=1$, we get $t_0=8$ in Lemma \ref{lemma-norm-L2-L1}. This results in the following corollary.

\begin{corollary}\label{corollary-L2-L1}
Under the conditions of Lemma \ref{lemma4.2}, one has
$
\left(\sum_{j=1}^q\|  w_{\Lambda_j} \|_\infty^2\right)^{1/2}
\leq\xi_q,
$
where
\begin{align}\label{corollary-L2-L1-2}
\xi_q=\left\{
\begin{array}{cl}
 1,&~\textrm{ if } ~ q=1,\\
\frac{2}{\sqrt{q}}+\frac{\sqrt{q}}{4},&~\textrm{ if } ~ 2\leq q< 8,\\
\sqrt{2},& ~\textrm{ if }~ q\geq 8,
\end{array}\right.
\end{align}
which is strictly  decreasing in the interval $[2,8]$  and $\underset{q\geq 1}{\max} ~\xi_q= \xi_2=\frac{5}{4}\sqrt{2}$.
\end{corollary}

Using Lemmas \ref{lemma4.2} and \ref{lemma-norm-L2-L1}, we can establish the next lemma.

\begin{lemma}\label{lemma4.3-right}
Let $x \in \mathbb{R}^n$ be a vector satisfying $y= Ax+\nu$ where  $\nu$ is a noise vector. Let $S=\mathcal{L}_k (x)$ and let $V \subseteq N $ be any given index set such that  $S \subseteq V$.
At the iterate $x^p$, the vectors $w^{(j)}, ~ j=1,\cdots,\omega,$ are generated by HBROT$\omega$ or  HBROTP$\omega$. For every $i\in\{1,\ldots,\omega\}$, we have
\begin{align}\label{lem4.3-right-0}
\Theta^i:= &\left\|A\left[(u^p-x_S)\circ w_H^{(i)}\right]_{\overline{V}}\right\|_2\nonumber\\
\leq &\sqrt{1+\delta_k} \Big[\big(\xi_q|1-\alpha+\beta|+2\alpha\delta_{3k}\big)\|x^p-x_S\|_2
+\beta \xi_q\|x^{p-1}-x_S\|_2  \nonumber\\
& +2\alpha\sqrt{1+\delta_k}\| \nu'\|_2\Big],
\end{align}
where  $w_H^{(i)}$ is the Hadamard product of vectors $w^{(j)}(j=1,\cdots,i)$, i.e.,
\begin{align}\label{Had-prod-w}
w_H^{(i)}:=w^{(1)}\circ w^{(2)}\circ \cdots \circ w^{(i)},\ \ i=1,\ldots, \omega,
\end{align}
 and $\xi_q$ is given by \eqref{corollary-L2-L1-2} with  $q=\lceil \frac{n-|V|}{k}\rceil$.
\end{lemma}

\begin{proof}
Taking $w=w^{(1)}$ and $\Lambda=\overline{V}$ in Lemma \ref{lemma4.2} and Corollary \ref{corollary-L2-L1}, one has
\begin{equation}\label{lem4.3-right-pf-2}
\sum_{j=1}^q\|  (w^{(1)})_{\Lambda_j} \|_\infty< 2, ~ \sqrt{\sum_{j=1}^q\|  (w^{(1)})_{\Lambda_j} \|_\infty^2}
\leq\xi_q,
\end{equation}
where   $q=\lceil \frac{n-|V|}{k}\rceil$ and the definition of   ${\Lambda_j}, ~j=1,\ldots,q, $  can be found in Lemma  \ref{lemma4.2}.
Next, we  derive the relation \eqref{lem4.3-right-0} for given $i\in\{1,\ldots,\omega\}$. Define the $k$-sparse vectors $z^{(l)}:=[(u^p-x_S)\circ w_H^{(i)}]_{\Lambda_l}, ~l=1,\ldots,q,$  where $u^p$ and $w_H^{(i)}$ are given by \eqref{algorithm-OTP-1} and \eqref{Had-prod-w}, respectively.  Since $w^{(j)}\in\mathcal{P}^k(j=1,\ldots,\omega)$, we have
\begin{equation}\label{lem4.3-right-pf-3}
\|z^{(l)}\|_2
\leq \|(w_H^{(i)})_{\Lambda_l}\|_\infty\cdot\|(u^p-x_S)_{\Lambda_l}\|_2
\leq \|( w^{(1)})_{\Lambda_l}\|_\infty\cdot\|(u^p-x_S)_{\Lambda_l}\|_2,
\end{equation}
where the second inequality follows from \eqref{Had-prod-w} and $0\leq w^{(j)} \leq \bf e$ for $j=1,\ldots,\omega$.
Since $z^{(l)}, ~l=1,\ldots,q, $ are  $k$-sparse vectors, from the definition of $\Theta^i$ in  \eqref{lem4.3-right-0} and  \eqref{lem4.3-right-pf-3}, we obtain
\begin{align}\label{lem4.3-right-pf-4}
\Theta^i &  =\left\|A\sum_{l=1}^{q}z^{(l)}\right\|_2\leq\sum_{l=1}^{q}\|Az^{(l)}\|_2
 \leq\sqrt{1+\delta_k}\sum_{l=1}^{q}\|z^{(l)}\|_2 \nonumber \\
 & \leq\sqrt{1+\delta_k}\sum_{l=1}^{q}\|( w^{(1)})_{\Lambda_l}\|_\infty\cdot\|(u^p-x_S)_{\Lambda_l}\|_2,
\end{align}
where the second inequality is given  by \eqref{def-RIC-1}. Since $|\Lambda_l|\leq k$ and $| \textrm{supp} (x^p-x_S)\cup \Lambda_l|\leq 3k$ for $l=1,\dots,q,$ by using \eqref{th-OTP-pf-1}  and Lemma \ref{lem-basic-ineq}, we have
\begin{align}\label{lem4.3-right-pf-5}
\|(u^p-x_S)_{\Lambda_l}\|_2\leq& |1-\alpha+\beta|\cdot\|(x^p-x_S)_{\Lambda_l}\|_2+\alpha\|[(I-A^ TA)(x^p-x_S)]_{\Lambda_l}\|_2\nonumber\\
&+\beta \|(x^{p-1}-x_S)_{\Lambda_l}\|_2+\alpha\|(A^T \nu')_{\Lambda_l}\|_2\nonumber\\
\leq & |1-\alpha+\beta|\cdot\|(x^p-x_S)_{\Lambda_l}\|_2+\alpha\delta_{3k}\|x^p-x_S\|_2\nonumber\\
&+\beta \|(x^{p-1}-x_S)_{\Lambda_l}\|_2+\alpha\sqrt{1+\delta_k}\| \nu'\|_2.
\end{align}
Substituting \eqref{lem4.3-right-pf-5} into  \eqref{lem4.3-right-pf-4} yields
\begin{align*}
\frac{\Theta^i}{\sqrt{1+\delta_k}}
\leq & |1-\alpha+\beta|\cdot\sum_{l=1}^{q}\|( w^{(1)})_{\Lambda_l}\|_\infty\cdot\|(x^p-x_S)_{\Lambda_l}\|_2 \\
& +\alpha\delta_{3k}\sum_{l=1}^{q}\|( w^{(1)})_{\Lambda_l}\|_\infty\cdot\|x^p-x_S\|_2\\
& +\beta \sum_{l=1}^{q}\|( w^{(1)})_{\Lambda_l}\|_\infty \cdot\|(x^{p-1}-x_S)_{\Lambda_l}\|_2 \\
& +\alpha\sqrt{1+\delta_k}\sum_{l=1}^{q}\|( w^{(1)})_{\Lambda_l}\|_\infty\cdot\| \nu'\|_2.
\end{align*}
It follows from Cauchy-Schwarz inequality and \eqref{lem4.3-right-pf-2} that
\begin{align*}
& \frac{\Theta^i}{\sqrt{1+\delta_k}} \\
 & \leq   |1-\alpha+\beta|\sqrt{\sum_{l=1}^q\|  (w^{(1)})_{\Lambda_l} \|_\infty^2}\sqrt{\sum_{l=1}^q\|(x^p-x_S)_{\Lambda_l}\|_2 ^2}+2\alpha\delta_{3k}\|x^p-x_S\|_2\\
 & ~~~ +\beta \sqrt{\sum_{l=1}^q\|  (w^{(1)})_{\Lambda_l} \|_\infty^2}\sqrt{\sum_{l=1}^q\|(x^{p-1}-x_S)_{\Lambda_l}\|_2 ^2}+2\alpha\sqrt{1+\delta_k} \| \nu'\|_2\\
& \leq  |1-\alpha+\beta|\xi_q \|x^p-x_S\|_2+2\alpha\delta_{3k}\|x^p-x_S\|_2
+\beta \xi_q\|x^{p-1}-x_S\|_2\\
  & ~~~ +2\alpha\sqrt{1+\delta_k}\| \nu'\|_2,
\end{align*}
where the last  inequality follows from the   relation
$
\sum_{j=1}^q\|z_{\Lambda_l}\|_2 ^2=\|z_{\overline{V}}\|_2 ^2\leq \|z\|_2 ^2 $  for any $ z\in \mathbb{R}^n
$
due to $\overline{V} = \bigcup_{j=1}^q {\Lambda_j} $ and $\Lambda_j\bigcap \Lambda_l=\emptyset$ for $ j\neq l$.
Thus \eqref{lem4.3-right-0} holds.  \hfill $\Box$
\end{proof}

We now estimate the term $\| y-A(u^p\circ w_H^{(\omega)})\|_2$ by using Lemma \ref{lemma4.3-right}.

\begin{lemma}\label{lemma4.3}
Let $x \in \mathbb{R}^n$ be a vector satisfying $y = Ax+\nu$ where  $\nu$ is a noise vector.
At the iterate $x^p$, the vectors $u^p$ and $w^{(j)}(j=1,\cdots,\omega)$ are generated by HBROT$\omega$ or  HBROTP$\omega$. Then
\begin{align}\label{lem4.3-right-2}
& \| y-A(u^p\circ w_H^{(\omega)})\|_2 \nonumber \\
 & \leq c_{1,q}\sqrt{1+\delta_k}\|x^p-x_S\|_2
+\sqrt{1+\delta_k}\beta\big[\xi_q(\omega-1)+1\big] \|x^{p-1}-x_S\|_2 \nonumber\\
 & ~~~ +\big[\alpha(2\omega-1)(1+\delta_k)+1\big] \| \nu'\|_2,
\end{align}
where $w_H^{(\omega)}$ is given by \eqref{Had-prod-w}, $ S:=\mathcal{L}_k (x)$, $q=\lceil  \frac{n-k}{k} \rceil $,
$\xi_q$ is given by \eqref{corollary-L2-L1-2} and $c_{1,q}$ is given as
\begin{equation}\label{lem4.3-right-3}
\begin{array}{rl}
c_{1,q}:=\big(\xi_q(\omega-1)+1\big)|1-\alpha+\beta|+ \alpha\big(2(\omega-1)\delta_{3k}+\delta_{2k}\big).
\end{array}
\end{equation}
\end{lemma}

\begin{proof}
 Let $\widehat{w}\in\mathcal{W}^k$ be a binary vector satisfying $ \textrm{supp} (x_S)\subseteq \textrm{supp} (\widehat{w})$ and $V=\textrm{supp} (\widehat{w})$.
From Lemma 4.3 in \cite{ZL21}, we get
\begin{align}\label{lem4.3-1}
  \| y-A[u^p\circ w_H^{(\omega)}] \|_2
  \leq
\| y-A(u^p\circ{\widehat{w}})\|_2+\sum_{i=1}^{\omega-1} \left\| A\left[(u^p-x_S)\circ w_H^{(i)}\circ ({\bf e}-\widehat{w})\right]\right\|_2,
\end{align}
where $w_H^{(i)}, ~i=1,\ldots,\omega, $ are given by \eqref{Had-prod-w}.
As $V= \textrm{supp} (\widehat{w})$ and $|V|=k$, it follows from \eqref{lem4.3-right-0} that
\begin{align}\label{lem4.3-right-1}
&   \left\|A\left[(u^p   -x_S)\circ w_H^{(i)}\circ ({\bf e}-\widehat{w})\right]\right\|_2
=\left\|A\left[(u^p-x_S)\circ w_H^{(i)}\right]_{\overline{\textrm{supp} (\widehat{w})}} \right\|_2\nonumber\\
& \leq \sqrt{1+\delta_k}\Big[(\xi_q|1-\alpha+\beta|+2\alpha\delta_{3k})\|x^p-x_S\|_2
+\beta \xi_q\|x^{p-1}-x_S\|_2 \nonumber \\
& ~~~ +2\alpha\sqrt{1+\delta_k}\| \nu'\|_2\Big],
\end{align}
where $q=\lceil  \frac{n-k}{k} \rceil $ and $i=1,\ldots,\omega-1$. We now estimate the term ${\|y-A(u^p\circ\widehat{w})\|}_2$ in \eqref{lem4.3-1}.
Because  $| \textrm{supp} (x^p-x_S)\cup  \textrm{supp} (\widehat{w})|\leq 2k$, by using \eqref{th-OTP-pf-3}  and Lemma \ref{lem-basic-ineq}, we obtain
\begin{align}
& \|(x_S-u^p)\circ\widehat{w}\|_2 \nonumber \\
& \leq  |1-\alpha+\beta|\cdot\|x^p-x_S\|_2
+\alpha \left\| \left[(I-A^ TA)(x^p-x_S)\right]_{\textrm{supp} (\widehat{w})} \right\|_2\nonumber\\
& ~~~ +\beta\| x^{p-1}-x_S\|_2+\alpha\left\| ( A^T \nu')_{\textrm{supp} (\widehat{w})} \right\|_2\nonumber\\
&  \leq  (|1-\alpha+\beta|+\alpha\delta_{2k})\|x^p-x_S\|_2
+\beta \|x^{p-1}-x_S\|_2+\alpha\sqrt{1+\delta_k}\| \nu'\|_2.\nonumber
\end{align}
It follows from \eqref{th-OTP-pf-2} that
\begin{align}\label{lem4.3-right-9}
 \|y-A(u^p\circ\widehat{w})\|_2  &  \leq  \sqrt{1+\delta_k} \Big[\big(|1-\alpha+\beta|+\alpha\delta_{2k}\big)\|x^p-x_S\|_2 \nonumber\\
& ~~~ +\beta \|x^{p-1}-x_S\|_2+\alpha\sqrt{1+\delta_k}\| \nu'\|_2\Big]+\| \nu'\|_2.
\end{align}
Combining \eqref{lem4.3-right-1}, \eqref{lem4.3-right-9} with \eqref{lem4.3-1} yields \eqref{lem4.3-right-2}.  \hfill $ \Box$
\end{proof}

The following property of the hard thresholding operator $\mathcal{H}_k(\cdot)$ is shown in \cite[Lemma 4.1]{ZL21}.
\begin{lemma} \emph{\cite{ZL21}}\label{lemma4.1}
Let $z,h\in\mathbb{R}^n$ be two vectors and  $\| h\|_0\leq k$. Then
\begin{equation*}\label{lemma4.1-1}
\|h-\mathcal{H}_k(z)\|_2\leq \|(z-h)_{S\cup S^*}\|_2+\|(z-h)_{S^*\setminus S}\|_2,
\end{equation*}
where $S:= \textrm{supp} (h)$ and $S^*:= \textrm{supp} (\mathcal{H}_k(z))$.
\end{lemma}

Let us state a fundamental property of  the orthogonal projection in the following lemma, which can be found in \cite[Eq.(3.21)]{F11} and \cite[p.49]{Z20}, and was extended to the general case in \cite[Lemma 4.2]{ZL20}.

\begin{lemma}\emph{\cite{F11,Z20,ZL20}} \label{lem-pursuit}
Let $x \in \mathbb{R}^n$ be a vector satisfying  $y = Ax+\nu$  where  $\nu$ is a  noise vector. Let $ S^*\subseteq N$ be an index set satisfying $|S^*|\leq k$ and
\begin{equation*}
z^*=\arg \min_{z\in \mathbb{R}^n} \{ {\| y-Az\|}_2^2:  \textrm{supp} (z)\subseteq S^*\}.
\end{equation*}
Then
\begin{equation*}\label{lem-pursuit-1}
\|z^*-x_S\|_2
\leq \frac{1}{\sqrt{1-(\delta_{2k})^2}}\| (z^*-x_S)_{\overline{S^*}}\|_2+\frac{ \sqrt{1+\delta_{k}}}{1-\delta_{2k}}\|\nu' \|_2,
\end{equation*}
where $S:=\mathcal{L}_k(x)$ and $\nu':=\nu+Ax_{\overline{S}}. $
\end{lemma}

We now establish the  error bounds for HBROT$\omega$ and HBROTP$\omega$.

\begin{theorem}\label{theorem-main}
Suppose that $n>3k$ and denote $\sigma:=\lceil\frac{n-2k}{k}\rceil$.  Let $x \in \mathbb{R}^n$ be a vector satisfying $y = Ax+\nu$ where $\nu$ is a noise vector.
Denote
\begin{equation}\label{th-main-0}
 t_k:=\frac{\sqrt{1+\delta_k}}{\sqrt{1-\delta_{2k}}}, \ \ z_k:=\sqrt{1-\delta^2_{2k}}.
\end{equation}
   \begin{itemize}\item [ \emph{(i)} ] Assume that the $(3k)$-th order RIC, $\delta_{3k}$, of the   matrix $A$ and the nonnegative parameters  $(\alpha, \beta)$  satisfy $\delta_{3k}<\gamma^*(\omega)$ and
 \begin{equation}\label{th-main-1}
  \beta<\frac{1-d_1}{1+ d_1+d_2},~  \frac{(d_0+d_2+2)\beta+d_0}{d_0-d_1+1}<\alpha
<\frac{d_0+2-(d_2-d_0)\beta}{d_0+ d_1+1},
  \end{equation}
  where $\gamma^*(\omega)$ is the unique root of the equation $G_{\omega}(\gamma)=1$ in the interval  $(0,1)$, where
  \begin{equation}\label{th-main-2}
   G_{\omega}(\gamma):=(2\omega+1)\gamma \sqrt{\frac{1+\gamma}{1-\gamma}}+\gamma,
  \end{equation}
and the constants  $ d_0, d_1, d_2$ are given as
   \begin{equation}\label{th-main-3}
 \left\{\begin{array}{l} d_0:=t_k (\omega\xi_\sigma+1), \\
 d_1:=t_k(2\omega\delta_{3k}+\delta_{2k})+\delta_{3k}, \\
 d_2:=t_k  [\xi_\sigma(\omega-1)+1]\frac{2\omega\delta_{3k}+\delta_{2k}}{2(\omega-1)\delta_{3k}+\delta_{2k}}.
 \end{array}
 \right.
   \end{equation}
Then, the sequence $\{x^p\}$  produced by HBROT$\omega$ obeys
   \begin{equation}\label{th-main-4}
       \|x^{p}-x_S\|_2\leq \theta_1^{p-1}\left[\|x^1-x_S\|_2+(\theta_1-b_1)\|x^0-x_S\|_2\right]+\frac{b_3}{1-\theta_1}\| \nu'\|_2
   \end{equation}
with $\theta_1:=\frac{b_1+\sqrt{b_1^2+4b_2}}{2}$. The fact $\theta_1< 1$ is ensured under \eqref{th-main-1} and $b_1,b_2,b_3$ are given as
   \begin{align}\label{th-main-6}
   b_1:=& t_k c_\sigma+\left(|1+\beta-\alpha|+\alpha \delta_{3k}\right),~~
   b_2:= \beta t_k [\xi_\sigma(\omega-1)+1]\frac{c_\sigma}{ c_{1,\sigma}}+\beta ,\nonumber\\
   b_3:=&\frac{\alpha(2\omega-1)(1+\delta_k)+2}{\sqrt{1-\delta_{2k}}}\cdot\frac{2\delta_{3k}}{2(\omega-1)\delta_{3k}
   +\delta_{2k}}\cdot\frac{c_\sigma}{c_\sigma- c_{1,\sigma}}+\alpha \sqrt{1+\delta_k},
    \end{align}
where $ c_{1,\sigma}$ and $\xi_\sigma$ are given by \eqref{lem4.3-right-3} and \eqref{corollary-L2-L1-2}, respectively, and
 \begin{equation}\label{th-main-7}
c_\sigma:= (\omega\xi_\sigma+1)|1-\alpha+\beta|+ \alpha(2\omega\delta_{3k}+\delta_{2k}).
 \end{equation}

\item [ \emph{(ii)} ]   Suppose that the $(3k)$-th order RIC, $\delta_{3k}$, of the matrix $A$ and the nonnegative parameters $(\alpha, \beta)$  satisfy $\delta_{3k}<\gamma^{\sharp}(\omega)$ and
  \begin{equation}\label{th-main-pursuit-1}
  \beta<\frac{z_k-d_1}{1+ d_1+d_2},~ \frac{(d_0+d_2+2)\beta+d_0+1-z_k}{d_0-d_1+1}<
\alpha<\frac{d_0+1+z_k-(d_2-d_0)\beta}{d_0+ d_1+1},
   \end{equation}
   where the constants $ d_0, d_1, d_2$ are given by \eqref{th-main-3} and $\gamma^{\sharp}(\omega)$ is the unique root of the equation $\frac{1}{\sqrt{1-\gamma^2}}G_{\omega}(\gamma)=1$ in the interval $(0,1)$, where $G_{\omega}(\gamma)$ is given by \eqref{th-main-2}. Then, the  sequence $\{x^p\}$  produced by HBROTP$\omega$ obeys
    \begin{align}\label{th-main-pursuit-2}
        \|x^{p}-x_S\|_2 & \leq \theta_2^{p-1}\left[\|x^1-x_S\|_2+(\theta_2-\frac{b_1}{z_k})\|x^0-x_S\|_2\right]  \nonumber \\
        & ~~~ +\frac{1}{1-\theta_2}\left(\frac{b_3}{z_k}+\frac{ \sqrt{1+\delta_{k}}}{1-\delta_{2k}}\right)\| \nu'\|_2
    \end{align}
with $\theta_2:=\frac{b_1+\sqrt{b_1^2+4b_2z_k}}{2 z_k}$. The fact $\theta_2< 1$ is ensured under \eqref{th-main-pursuit-1}  and the constants $b_i(i=1,2,3)$  and $z_k$ are given by \eqref{th-main-6} and \eqref{th-main-0}, respectively.
\end{itemize}
\end{theorem}

\begin{proof}
Let  $x^{\sharp}=\mathcal{H}_k(u^p\circ w_H^{(\omega)})$ be generated by the Algorithms, where $w_H^{(\omega)}$ is given by \eqref{Had-prod-w}. By using Lemma \ref{lemma4.1}, we  have
\begin{equation}\label{th-main-pf-1}
\|x_S-x^{\sharp}\|_2\leq \|(u^p\circ w_H^{(\omega)}-x_S)_{X\cup S}\|_2+\|(u^p\circ w_H^{(\omega)}-x_S)_{X\setminus S}\|_2,
\end{equation}
where $X= \textrm{supp} (x^{\sharp})$.  Using  \eqref{th-OTP-pf-1} and the triangle inequality, we have that
\begin{align*}
& \|(u^p\circ w_H^{(\omega)}-x_S)_{X\setminus S} \|_2  = \|[(u^p-x_S)\circ w_H^{(\omega)}]_{X\setminus S} \|_2
\leq \|(u^p-x_S)_{X\setminus S} \|_2\\
&  \leq |1-\alpha+\beta|\cdot \left\|(x^p-x_S)_{X\setminus S} \right\|_2
+\alpha \|[(I-A^ TA)(x^p-x_S)]_{X\setminus S} \|_2\nonumber\\
&  ~~~ +\beta \|(x^{p-1}-x_S)_{X\setminus S} \|_2
+\alpha  \|(A^T \nu')_{X\setminus S} \|_2,\nonumber
\end{align*}
where the first equality is ensured by $(x_S)_{X\backslash S}=0$  and the first inequality is due to  \eqref{Had-prod-w} and $0\leq w^{(j)} \leq {\bf e } $ for $ j=1,\ldots,\omega.$ Since $|X\setminus S|\leq k$ and $| \textrm{supp} (x^p-x_S)\cup (X\setminus S)|\leq 3k$, by using Lemma \ref{lem-basic-ineq},  we see that
\begin{align}\label{th-main-pf-3}
\|(u^p\circ w_H^{(\omega)}-x_S)_{X\setminus S} \|_2
& \leq (|1+\beta-\alpha|  +\alpha \delta_{3k})\|x^p-x_S\|_2 \nonumber \\
& ~~ +
\beta\| x^{p-1}-x_S\|_2
  +\alpha \sqrt{1+\delta_k}\| \nu'\|_2.
\end{align}
Denote
\begin{equation}\label{def-Theta12}
\Theta_1:= \|A(u^p\circ w_H^{(\omega)}-x_S)_{X\cup S} \|_2,\ \
\Theta_2:=   \|A(u^p\circ w_H^{(\omega)}-x_S)_{\overline{X\cup S}} \|_2.\ \
\end{equation}
 As $|X\cup S|\leq 2k$, by using \eqref{def-RIC-1},  we obtain
 \begin{equation}\label{Theta1-lowbd}
 \Theta_1\geq\sqrt{1-\delta_{2k}} \|(u^p\circ w_H^{(\omega)}-x_S)_{X\cup S} \|_2.
 \end{equation}
For  any given  $\zeta\in (0,1)$, we  consider the following two cases associated with $\Theta_1$ and $\Theta_2$.

\textbf{Case 1. } $\Theta_2\leq\zeta\Theta_1$. Since $y=Ax_S+\nu'$, by   the triangle inequality and \eqref{def-Theta12}, we have
\begin{align}\label{th-main-pf-case1-1}
& \|y-A(u^p\circ w_H^{(\omega)})\|_2=  \|A(u^p\circ w_H^{(\omega)}-x_S)-\nu'\|_2\nonumber\\
  &= \|A(u^p\circ w_H^{(\omega)}-x_S)_{X\cup S}+A(u^p\circ w_H^{(\omega)}-x_S)_{\overline{X\cup S}}-\nu' \|_2\nonumber\\
 & \geq  \Theta_1-\Theta_2-\|\nu'\|_2\nonumber\\
 & \geq  (1-\zeta)\Theta_1-\|\nu'\|_2.
\end{align}
Merging \eqref{Theta1-lowbd},  \eqref{th-main-pf-case1-1} with \eqref{lem4.3-right-2} yields
 \begin{align}\label{th-main-pf-case1-2}
 & \|(u^p\circ w_H^{(\omega)}-x_S)_{X\cup S}  \|_2  \nonumber \\
  &  \leq \frac{1}{(1-\zeta)\sqrt{1-\delta_{2k}}} ( \|y-A(u^p\circ w_H^{(\omega)})\|_2+\|\nu'\|_2 )\nonumber\\
& \leq  \frac{t_k c_{1,{q_1}}}{1-\zeta}\|x^p-x_S\|_2
+\frac{\beta t_k}{1-\zeta}\big[\xi_{q_1}(\omega-1)+1\big] \|x^{p-1}-x_S\|_2 \nonumber \\
 & ~~ ~ +\frac{\alpha(2\omega-1)(1+\delta_k)
  +2}{(1-\zeta)\sqrt{1-\delta_{2k}}}\| \nu'\|_2,
 \end{align}
 where $q_1=\lceil\frac{n-k}{k}\rceil=\sigma+1$ and $t_k, c_{1,{q_1}}$ are  given in \eqref{th-main-0} and \eqref{lem4.3-right-3}, respectively.

 \textbf{Case 2. } $\Theta_2>\zeta\Theta_1$. From \eqref{def-Theta12} and \eqref{Theta1-lowbd}, we obtain
  \begin{equation}\label{th-main-pf-case2-1}
  \|(u^p\circ w_H^{(\omega)}-x_S)_{X\cup S} \|_2 \leq   \frac{1}{\zeta\sqrt{1-\delta_{2k}}} \|A[(u^p-x_S)\circ w_H^{(\omega)}]_{\overline{X\cup S}}  \|_2.
  \end{equation}
 Taking $V=X\cup S$ and $i=\omega$ in \eqref{lem4.3-right-0}, one has
  \begin{align}\label{th-main-pf-case2-2}
&  \|A[(u^p-x_S)\circ w_H^{(\omega)}]_{\overline{X\cup S}}  \|_2\nonumber\\
& \leq    \sqrt{1+\delta_k} \left[  c_{2,{q_2}}\|x^p-x_S\|_2
  +\beta \xi_{q_2}\|x^{p-1}-x_S\|_2+2\alpha\sqrt{1+\delta_k}\| \nu'\|_2\right]
  \end{align}
with  $q_2=\lceil\frac{n-|X\cup S|}{k}\rceil\geq \sigma$ which is  due to   $|X\cup S|\leq2k$, and $c_{2,{q_2}}$ is defined as
  \begin{equation}\label{th-main-pf-case2-2-1}
  c_{2,{q_2}}:=\xi_{q_2}|1-\alpha+\beta|+2\alpha\delta_{3k}.
   \end{equation}
  Substituting \eqref{th-main-pf-case2-2} into \eqref{th-main-pf-case2-1}, we get
    \begin{align}\label{th-main-pf-case2-3}
 \|(u^p\circ w_H^{(\omega)}-x_S)_{X\cup S} \|_2
  & \leq   \frac{t_k}{\zeta}  [  c_{2,{q_2}}\|x^p-x_S\|_2
    +\beta \xi_{q_2}\|x^{p-1}-x_S\|_2  \nonumber \\
    & ~~~  +2\alpha\sqrt{1+\delta_k}\| \nu'\|_2 ].
    \end{align}
From \eqref{corollary-L2-L1-2}, we see that $\xi_q$ is decreasing in $[2,n]$. For $q_1=\sigma+1$ and $q_2\geq \sigma\geq 2$, we have $\xi_{q_1},\xi_{q_2}\leq \xi_\sigma$.  It follows from \eqref{lem4.3-right-3} and \eqref{th-main-pf-case2-2-1} that $c_{1,{q_1}}\leq c_{1,\sigma}$ and $c_{2,{q_2}}\leq c_{2,\sigma}$. Combining  \eqref{th-main-pf-case1-2} and \eqref{th-main-pf-case2-3} leads to
 \begin{align}\label{th-main-pf-4}
&  \|(u^p\circ w_H^{(\omega)}-x_S)_{X\cup S} \|_2\nonumber\\
 & \leq    t_k\max\left\{\frac{ c_{1,\sigma}}{1-\zeta},\frac{  c_{2,\sigma}}{\zeta}\right\}\|x^p-x_S\|_2\nonumber\\
& ~~~ +\beta t_k\max\left\{\frac{\xi_\sigma(\omega-1)+1}{1-\zeta},\frac{\xi_\sigma}{\zeta}\right\}\|x^{p-1}-x_S\|_2\nonumber\\
 & ~~~ +\frac{1}{\sqrt{1-\delta_{2k}}}\max\left\{\frac{\alpha(2\omega-1)(1+\delta_k)
  +2}{1-\zeta},\frac{2\alpha(1+\delta_k)}{\zeta}\right\}\| \nu'\|_2
 \end{align}
  for any $\zeta\in (0,1)$.

Next, we   select a suitable parameter $\zeta\in (0,1)$ such that the right hand of \eqref{th-main-pf-4} is as small as possible.
 For $\delta_{2k}\leq\delta_{3k}$ and $\xi_\sigma<2$  in \eqref{corollary-L2-L1-2}, we have
 \begin{equation}\label{th-main-pf-case2-5}
 \frac{ c_{2,\sigma}}{ c_{1,\sigma}}=\frac{\xi_\sigma|1-\alpha+\beta|+2\alpha\delta_{3k}}{[\xi_\sigma(\omega-1)+1]|1-\alpha+\beta|+\alpha[2(\omega-1)\delta_{3k}+\delta_{2k}]}
 \leq\frac{2\delta_{3k}}{2(\omega-1)\delta_{3k}+\delta_{2k}}.
 \end{equation}
 It is easy to check that
 \begin{equation}\label{th-main-pf-5}
 \min_{\zeta\in (0,1)}\max\left\{\frac{ c_{1,\sigma}}{1-\zeta},\frac{  c_{2,\sigma}}{\zeta}\right\}= c_{1,\sigma}+c_{2,\sigma}
 =c_\sigma,
 \end{equation}
 where $c_\sigma$ is given by \eqref{th-main-7} and its minimum attains at
 \begin{equation}\label{th-main-pf-6}
\zeta^*=\frac{ c_{2,\sigma}}{ c_{1,\sigma}+ c_{2,\sigma}}=\frac{\xi_\sigma|1-\alpha+\beta|+2\alpha\delta_{3k}}{(\omega\xi_\sigma+1\big)|1-\alpha+\beta|+ \alpha(2\omega\delta_{3k}+\delta_{2k})}.
  \end{equation}
That is,
    \begin{equation}\label{add-zeta-1}
   \max\left\{\frac{ c_{1,\sigma}}{1-\zeta^*},\frac{  c_{2,\sigma}}{\zeta^*}\right\}=c_\sigma.
 \end{equation}
Moreover, noting that $\xi_\sigma<2$ and $\delta_{2k}\leq \delta_{3k}$, we have $\zeta^*\geq\frac{\xi_\sigma}{\omega\xi_\sigma+1}$.
  In particular, by taking  $\zeta=\zeta^*$ in (\ref{th-main-pf-4}), we deduce that
  \begin{equation*}\label{th-main-pf-6-1}
  \max\left\{\frac{\xi_\sigma(\omega-1)+1}{1-\zeta^*},\frac{\xi_\sigma}{\zeta^*}\right\}=\frac{\xi_\sigma(\omega-1)+1}{1-\zeta^*}=\frac{\xi_\sigma(\omega-1)+1}{ c_{1,\sigma}}c_\sigma,
   \end{equation*}
   and
    \begin{align}\label{th-main-pf-6-2}
 & \max\left\{\frac{\alpha(2\omega-1)(1+\delta_k)
     +2}{1-\zeta^*}, ~\frac{2\alpha(1+\delta_k)}{\zeta^*}\right\}  \nonumber\\
 & =  \frac{1}{\zeta^*}\max\left\{[\alpha(2\omega-1)(1+\delta_k)
     +2]\frac{\zeta^*}{1-\zeta^*}, ~2\alpha(1+\delta_k)\right\} \nonumber\\
  &  = \frac{c_\sigma}{c_{2,\sigma}}\max\left\{[\alpha(2\omega-1)(1+\delta_k)
     +2]\frac{c_{2,\sigma}}{c_{1,\sigma}}, ~2\alpha(1+\delta_k)\right\} \nonumber\\
   &  \leq
     \frac{c_\sigma}{c_{2,\sigma}}\max\left\{[\alpha(2\omega-1)(1+\delta_k)
     +2]\frac{2\delta_{3k}}{2(\omega-1)\delta_{3k}+\delta_{2k}}, ~2\alpha(1+\delta_k)\right\}\nonumber \\
 &  =   \Big[\alpha(2\omega-1)(1+\delta_k)+2\Big]\frac{2\delta_{3k}}{2(\omega-1)\delta_{3k}+\delta_{2k}}\cdot\frac{c_\sigma}{ c_{2,\sigma}},
    \end{align}
 where the second equality is given by \eqref{th-main-pf-6}, the inequality above follows from \eqref{th-main-pf-case2-5},  and the last equality holds owing to $\delta_{2k}\leq \delta_{3k}$.
Merging \eqref{th-main-pf-4} with \eqref{add-zeta-1}-\eqref{th-main-pf-6-2},  we obtain
  \begin{align}\label{th-main-pf-7}
&  \|(u^p\circ w_H^{(\omega)}-x_S)_{X\cup S} \|_2 \nonumber \\
  & \leq   t_kc_\sigma \|x^p-x_S\|_2
+\beta t_k \frac{\xi_\sigma(\omega-1)+1}{ c_{1,\sigma}}c_\sigma \|x^{p-1}-x_S\|_2\nonumber\\
  & ~~~ +\frac{\alpha(2\omega-1)(1+\delta_k)+2}{\sqrt{1-\delta_{2k}}} \cdot \frac{2\delta_{3k}}{2(\omega-1)\delta_{3k}+\delta_{2k}}\cdot\frac{c_\sigma}{ c_{2,\sigma}} \| \nu'\|_2.
  \end{align}
  Combining \eqref{th-main-pf-3}, \eqref{th-main-pf-7} with  \eqref{th-main-pf-1}, we have
  \begin{equation}\label{th-main-pf-8}
   \|x_S-x^{\sharp}\|_2\leq b_1\|x^p-x_S\|_2+b_2\|x^{p-1}-x_S\|_2+b_3\| \nu'\|_2,
  \end{equation}
  where the constants $b_1,b_2,b_3$ are given by \eqref{th-main-6}.

Next, we   estimate $    \|x^{p+1}-x_S\|_2$ for HBROT$\omega$  and HBROTP$\omega$  based on the relation \eqref{th-main-pf-8}.

(i) Since $x^{p+1}=x^{\sharp}$ in HBROT$\omega$, \eqref{th-main-pf-8} becomes
   \begin{equation}\label{th-main-pf-10}
    \|x^{p+1}-x_S\|_2\leq b_1\|x^p-x_S\|_2+b_2\|x^{p-1}-x_S\|_2+b_3\| \nu'\|_2.
   \end{equation}
Now, we consider the conditions of Lemma  \ref{lem-two-level-geometric}.
   Merging \eqref{th-main-pf-case2-5} with \eqref{th-main-pf-5} produces
     \begin{equation*}
 \frac{ c_\sigma}{ c_{1,\sigma}}\leq \frac{2\omega\delta_{3k}+\delta_{2k}}{2(\omega-1)\delta_{3k}+\delta_{2k}}.
    \end{equation*}
   It follows from \eqref{th-main-6} and  \eqref{th-main-7} that
   \begin{align}\label{th-main-pf-12}
   b_1+ b_2 & \leq t_k c_\sigma+\left(|1+\beta-\alpha|+\alpha \delta_{3k}\right) \nonumber \\
   & ~~~ + \left\{t_k  \frac{2\omega\delta_{3k}+\delta_{2k}}{2(\omega-1)\delta_{3k}+\delta_{2k}}[\xi_\sigma(\omega-1)+1]+1\right\}\beta=F(\alpha,\beta),
    \end{align}
    where
    \begin{align}\label{def-F}
    F(\alpha,\beta):=&(d_0+1)|1-\alpha+\beta|+ d_1\alpha+ (d_2+1)\beta,\nonumber\\
    =&\left\{
\begin{array}{ll}
-(d_0-d_1+1)\alpha+(d_0+d_2+2)\beta+d_0+1, &\textrm{ if }\ \alpha\leq 1+\beta,\\
(d_0+ d_1+1)\alpha+(d_2-d_0)\beta-(d_0+1),&\textrm{ if }\ \alpha> 1+\beta,
\end{array}
\right.
\end{align}
with  the constants $ d_0, d_1, d_2$ are given by \eqref{th-main-3}.

 Based on the fact $\delta_k\leq \delta_{2k}\leq \delta_{3k}<\gamma^*(\omega)$ and  the function $G_{\omega}(\gamma)$  in \eqref{th-main-2} is strictly increasing in the interval $(0,1)$,  by using \eqref{th-main-3}, we have
        \begin{equation}\label{th-main-pf-14-1}
  d_1\leq[(2\omega+1)t_k+1] \delta_{3k}\leq G_{\omega}(\delta_{3k})< G_{\omega}(\gamma^*(\omega))=1,
        \end{equation}
   which shows that the range of  $\beta$  in \eqref{th-main-1} is well defined.
 From the first inequality in \eqref{th-main-1}, we see that
 \begin{equation}\label{th-main-pf-15}
 \frac{(d_0+d_2+2)\beta+d_0}{d_0-d_1+1}<
1+\beta<\frac{d_0+2-(d_2-d_0)\beta}{d_0+ d_1+1},
  \end{equation}
     which implies that the range of $\alpha$  in \eqref{th-main-1} is also  well defined.
  Merging \eqref{def-F}-\eqref{th-main-pf-15} with the second inequality in \eqref{th-main-1}, we see that if  $  \frac{(d_0+d_2+2)\beta+d_0}{d_0-d_1+1}<\alpha\leq 1+\beta, $ then
 $$ F(\alpha,\beta)
< -(d_0-d_1+1)\frac{(d_0+d_2+2)\beta+d_0}{d_0-d_1+1}+(d_0+d_2+2)\beta+d_0+1= 1,$$ and if $ 1+\beta<\alpha<\frac{d_0+2-(d_2-d_0)\beta}{d_0+ d_1+1},$ then
 $$ F(\alpha,\beta) < (d_0+ d_1+1)\frac{d_0+2-(d_2-d_0)\beta}{d_0+ d_1+1}+(d_2-d_0)\beta-(d_0+1) =1.$$
It follows from \eqref{th-main-pf-12} that $b_1+b_2<1$. Hence, applying Lemma \ref{lem-two-level-geometric} to the relation \eqref{th-main-pf-10}, we conclude that \eqref{th-main-4} holds with $\theta_1=\frac{b_1+\sqrt{b_1^2+4b_2}}{2}<1$.

  (ii)  Since $x^{p+1}$ is given by \eqref{algorithm-ROTP-3} and $S^{p+1}= \textrm{supp} (x^{\sharp})$ in HBROTP$\omega$, by setting $S^*=S^{p+1}$ and $z^*=x^{p+1}$  in Lemma \ref{lem-pursuit}, we have
  \begin{align}\label{th-main-pf-pt2-1}
  \|x^{p+1}-x_S\|_2 & \leq \frac{1}{z_k}\left\|(x^{\sharp}-x_S)_{\overline{S^{p+1}}} \right \|_2+\frac{ \sqrt{1+\delta_{k}}}{1-\delta_{2k}}\| \nu' \|_2
 \nonumber\\
 & \leq \frac{1}{z_k}\|x^{\sharp}-x_S\|_2+\frac{ \sqrt{1+\delta_{k}}}{1-\delta_{2k}}\| \nu' \|_2,
  \end{align}
  where $z_k$ is given in \eqref{th-main-0} and the first inequality follows from the fact $(x^{p+1})_{\overline{S^{p+1}}}=(x^{\sharp})_{\overline{S^{p+1}}}=0$.
  Combining \eqref{th-main-pf-pt2-1} with \eqref{th-main-pf-8}, we have
  \begin{align}\label{th-main-pf-pt2-2}
\|x^{p+1}-x_S\|_2 & \leq \frac{b_1}{z_k}\|x^p-x_S\|_2+\frac{b_2}{z_k}\|x^{p-1}-x_S\|_2+\left(\frac{b_3}{z_k}+\frac{ \sqrt{1+\delta_{k}}}{1-\delta_{2k}}\right)\| \nu' \|_2.
 \end{align}
Similar to the analysis in Part (i), we need to show that $ \frac{b_1}{z_k}+ \frac{b_2}{z_k}<1$.

From the conditions of Theorem \ref{theorem-main}(ii), we have $\delta_{2k}\leq \delta_{3k}<\gamma^{\sharp}(\omega)$. Since the function $G_{\omega}(\gamma)$ in \eqref{th-main-2} is strictly increasing in $(0,1)$,
one has
 \begin{equation*}\label{th-main-pf-pt2-3}
   d_1\leq G_{\omega}(\delta_{3k})< G_{\omega}(\gamma^{\sharp}(\omega))=\sqrt{1-(\gamma^{\sharp}(\omega))^2}<\sqrt{1-(\delta_{2k})^2}=z_k,
 \end{equation*}
   where the first inequality is given by \eqref{th-main-pf-14-1}, the first equality follows from the fact that $\gamma^{\sharp}(\omega)$ is the  root of $\frac{1}{\sqrt{1-\gamma^2}}G_{\omega}(\gamma)=1$ in $(0,1)$ and the last equality is given by \eqref{th-main-0}.
 It follows that the range of  $\beta$  in  \eqref{th-main-pursuit-1} is well defined.
  From the first  inequality in \eqref{th-main-pursuit-1}, we derive
  \begin{equation}\label{th-main-pf-pt2-4}
  \frac{(d_0+d_2+2)\beta+d_0+1-z_k}{d_0-d_1+1}<
 1+\beta<\frac{d_0+1+z_k-(d_2-d_0)\beta}{d_0+ d_1+1},
   \end{equation}
   which means that the range of  $\alpha$ in \eqref{th-main-pursuit-1} is well defined.
   Combining \eqref{def-F}, \eqref{th-main-pf-pt2-4} with the second inequality in \eqref{th-main-pursuit-1} leads to
 \begin{align*}
F(\alpha,\beta)
<&\left\{
\begin{array}{ll}
-(d_0-d_1+1)  \frac{(d_0+d_2+2)\beta+d_0+1-z_k}{d_0-d_1+1}+(d_0+d_2+2)\beta+d_0+1, \\
 ~~~~~ ~~~~~ ~~~~~ ~~~~~\textrm{ if }   ~ \frac{(d_0+d_2+2)\beta+d_0+1-z_k}{d_0-d_1+1}<\alpha\leq 1+\beta,\\\\
(d_0+ d_1+1)\frac{d_0+1+z_k-(d_2-d_0)\beta}{d_0+ d_1+1}+(d_2-d_0)\beta-(d_0+1), \\
~~~~~ ~~~~~ ~~~~~ ~~~~~\textrm{ if }  ~ 1+\beta<\alpha<\frac{d_0+1+z_k-(d_2-d_0)\beta}{d_0+ d_1+1},
\end{array}
\right.\\
=& z_k.\nonumber
 \end{align*}
 It follows from \eqref{th-main-pf-12} that $ \frac{b_1}{z_k}+ \frac{b_2}{z_k}<1$.  Therefore, by Lemma \ref{lem-two-level-geometric},  it follows from \eqref{th-main-pf-pt2-2} that \eqref{th-main-pursuit-2} holds with $\theta_2=\frac{b_1+\sqrt{b_1^2+4b_2z_k}}{2 z_k}<1$.  \hfill $ \Box$
\end{proof}

\begin{remark}

   \begin{itemize}
   \item [ \emph{(i)} ] When $\nu=0$ and $x$ is a $k$-sparse vector,  from \eqref{th-main-4} and \eqref{th-main-pursuit-2}, we observe that the sequence $\{ x^p\}$ generated by  HBROT$\omega$ or HBROTP$\omega$  converges to $x$.

       \item [ \emph{(ii)} ] The condition $n>3k$ in Theorem \ref{theorem-main} can be removed. If so, the constant $\xi_\sigma$ will be replaced by $\underset{q\geq 1}{\max} ~\xi_q=\frac{5}{4}\sqrt{2}$ (see Corollary \ref{corollary-L2-L1}). In addition, if $n>9k$, then $\sigma=\lceil\frac{n-2k}{k}\rceil\geq 8$. In this case,  we see from  \eqref{corollary-L2-L1-2} that $\xi_\sigma$ in Theorem \ref{theorem-main} can be replaced by $\underset{q\geq 2}{\min} ~\xi_q=\sqrt{2}$.

    \item [ \emph{(iii)} ] When $\omega=1$, HBROT$\omega$ and HBROTP$\omega$ reduce to HBROT and HBROTP, respectively. In this case, the RIP  bounds   in Theorem \ref{theorem-main} are reduced to  $\delta_{3k}<\gamma^*(1)\approx 0.2118$ for HBROT and $\delta_{3k}<\gamma^{\sharp}(1)\approx0.2079$ for HBROTP.

     \item [ \emph{(iv)} ]   It is not convenient to calculate the RIC of the   matrix A and \eqref{th-main-pursuit-1} is just a sufficient condition for the theoretical performance of HBROTP$\omega$.  In practical implementation,   the parameters $(\alpha,\beta)$ in HBROTP may be set as $ 0\leq \beta <1/4$ and $ \alpha \geq 1+\beta
         $ for simplicity to roughly meet the conditions \eqref{th-main-pursuit-1}.
   \end{itemize}

\end{remark}

\section{Numerical experiments}\label{simulation}

Sparse signal and image recovery through measurements $ y = Ax+\nu,$ where $ x$ denotes the signal/image to recover, is a typical linear inverse problem. In this section, we provide some experiment results for the proposed HBROTP algorithm  and compare its performance with several existing methods.
The  experiments  in Sections \ref{phase_trans} and \ref{real_image} are performed on a server with the processor Intel(R) Xeon(R)  CPU E5-2680 v3@ 2.50GHz  and 256GB memory, while others are performed  on a PC with the processor Intel(R) Core(TM) i7-10700 CPU
@ 2.90 GHz and 16 GB memory. All involved convex optimization problems  are solved by CVX \cite{GB17}  with solver \emph{`Mosek'} \cite{AA20}.  The comparison of six algorithms including HBROTP, ROTP2, PGROTP, $\ell_1$-min, OMP and PLB is mainly made via the phase transitions based on synthetic data together with the reconstruction, deblurring and denoising of a few real images.

\subsection{Phase transition}\label{phase_trans}
The first experiment is carried out to compare the performances of the algorithms except PLB through the phase transition curve (PTC)  \cite{BT15,BTW15} and average recovery time. All sparse vectors $x^*\in\mathbb{R}^{ n}$ and  matrices $A\in\mathbb{R}^{m\times n}$ are  randomly generated,  and the position of nonzero elements of $ x^*$ follows the uniform distribution.  In addition, all columns of  $A$   are  normalized and the entries of $A$ and the nonzeros of $x^*$  are independent and identically distributed random variables following ${\mathcal N}(0,1)$.   In this experiment, we  consider both  accurate measurements $y=Ax^*$ and  inaccurate measurements $y=Ax^*+\epsilon h$ with fixed $n=1000$,  where $\epsilon=5\times 10^{-3}$ is the noise level and $h\in\mathbb{R}^{ m}$ is a normalized standard Gaussian noise. We let  HBROTP start from $x^1=x^0= 0$ with fixed parameters $\alpha=5$ and $ \beta = 0.2$, while other algorithms start from  $x^0= 0.$  The maximum number of iterations of HBROTP, ROTP2 and PGROTP is set as 50, while  OMP is performed exactly $k$ iterations and $\ell_1$-min is performed by the solver  \emph{`Mosek'} directly. Given the random data $(A,x^*)$ or $(A,x^*,h),$   the recovery is counted as \emph{`success'} when the  criterion
\begin{equation*}\label{recov-criter}
\|x^p-x^*\|_2/\|x^*\|_2\leq 10^{-3}
\end{equation*}
is satisfied, in which $x^p$ is the solution generated by algorithms.

Denote by $\kappa=m/n$ and $\rho=k/m$, where $\kappa$ is often called the sampling rate or the compression ratio. In the $(\kappa, \rho)$-space, the region below the PTC  is  called the \emph{`success'} recovery region,  where the solution of the SLI problem can be exactly or approximately recovered, while the region above the PTC corresponds to the \emph{`failure'} region. Thus if the region below the PTC is wider, the performance of an algorithm would be better. We now briefly describe the mechanism for plotting the PTC  which is taken as  the classical 50\% logistic regression curve, and more detailed information can be found in  \cite{BT15,BTW15}. To generate the PTCs, 13 groups of   $m=\lceil \kappa\cdot n\rceil$ are considered, where the sampling rate $\kappa$ is ranged from 0.1 to 0.7 with stepsize 0.05. For any given $m$,  by using the  bisection method,  the approximated recovery phase transition region $[k_{\min},k_{\max}]$ is  produced for each algorithm, in which  the success rate of recovery is at least 90\%   as $k<k_{\min}$ and at most 10\%  as $k>k_{\max}$. The interval $[k_{\min},k_{\max}]$ will be equally divided into $\min\{k_{\max}-k_{\min},50\}$ parts, and 10 problem instances are tested for each $k$ to produce the  recovery success rate for given algorithm.  Thus the PTCs can be obtained from the logistic regression model in \cite{BT15,BTW15} directly.

\begin{figure}[htbp]
\subfigure[Accurate measurements]{
\begin{minipage}[t]{0.45\linewidth}
\centering
  \includegraphics[width=\textwidth,height=0.8\textwidth]{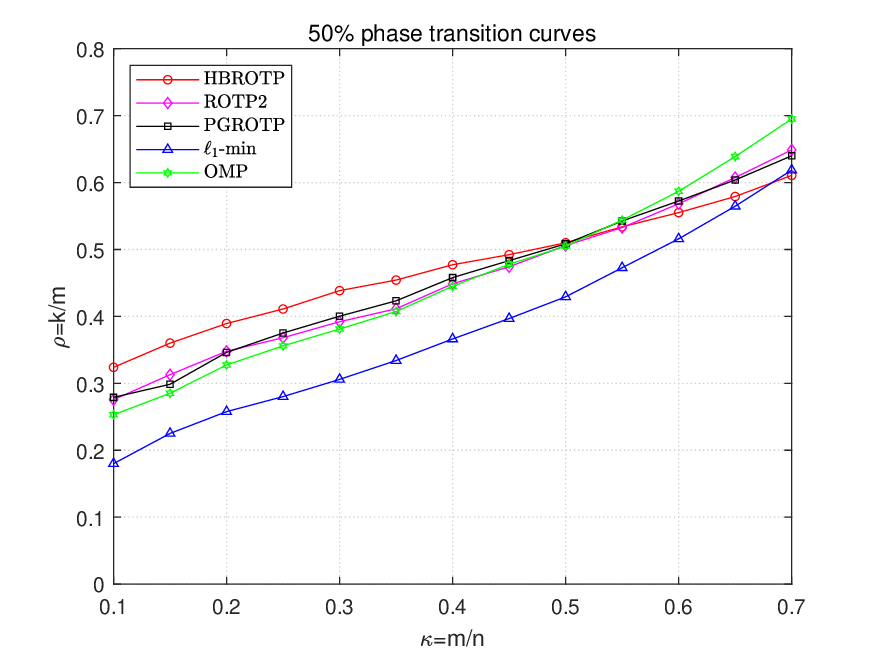}
   \end{minipage}
   }
   \subfigure[Inaccurate measurements ($\epsilon=5\times 10^{-3}$)]{
\begin{minipage}[t]{0.45\linewidth}
\centering
  \includegraphics[width=\textwidth,height=0.8\textwidth]{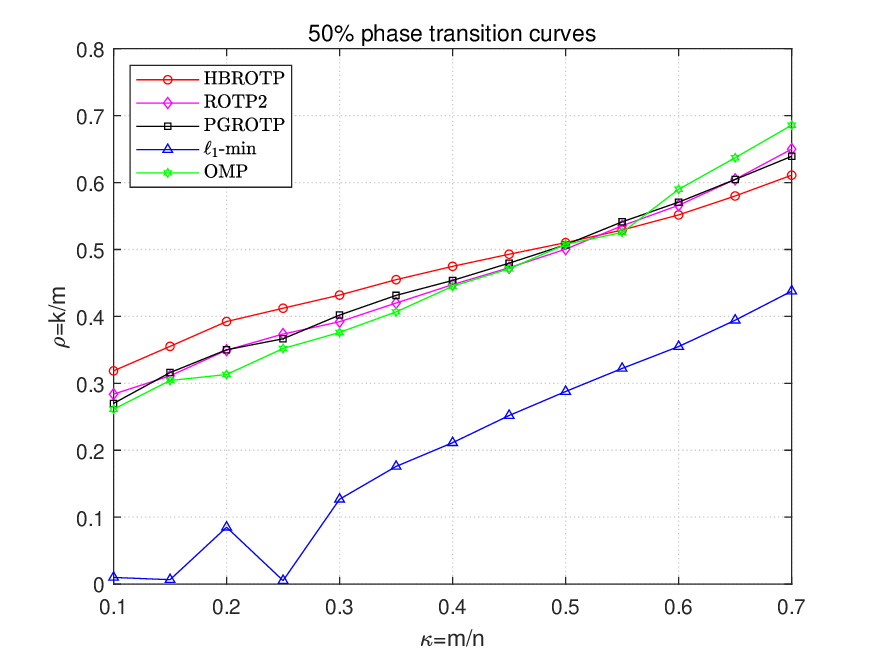}
   \end{minipage}
   }
   \caption{The $50 \%$ success rate phase transition curves for  algorithms.  }   \label{PTC-Map}
   \end{figure}

   The PTCs for the experimented algorithms are shown in  Fig. \ref{PTC-Map}(a) and (b), which correspond to the accurate measurements and inaccurate measurements with the noise level $\epsilon=5\times 10^{-3}$, respectively. The results indicate that HBROTP  has the highest PTC as $\kappa\leq 0.5,$ in which case  the recovery capability of HBROTP is superior to other algorithms in this experiment. However, the PTCs indicate that  ROTP2, PGROTP and OMP may perform relatively better than HBROTP with a larger $\kappa$. The comparison in Fig. \ref{PTC-Map}(a) and (b) demonstrates that all algorithms are robust for signal recovery when the measurements are slightly inaccurate except $\ell_1$-min. The comparison indicates that the overall performance of HBROTP is very comparable to those existing methods in this experiment.

   \begin{figure}[h]
   \centering
  \subfigure[Time  ratio of  ROTP2 and HBROTP]{
   \begin{minipage}[t]{0.45\linewidth}
   \centering
  \includegraphics[width=\textwidth,height=0.8\textwidth]{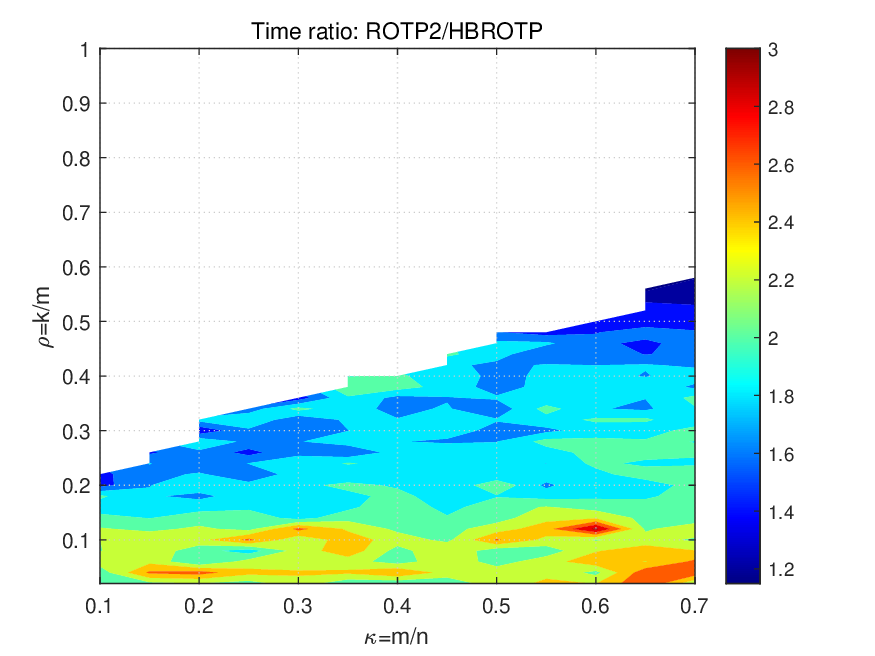}
  \end{minipage}
    }
   \subfigure[Time ratio of PGROTP and HBROTP]{
   \begin{minipage}[t]{0.45\linewidth}
  \centering
 \includegraphics[width=\textwidth,height=0.8\textwidth]{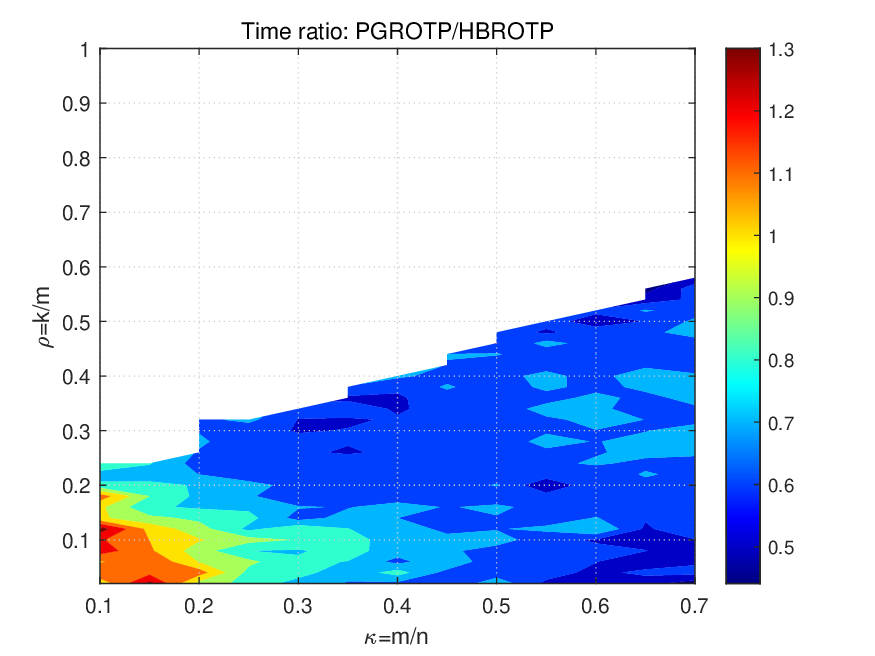}
    \end{minipage}
    }
    \subfigure[Time  ratio of  $\ell_1$-min and HBROTP]{
   \begin{minipage}[t]{0.45\linewidth}
   \centering
  \includegraphics[width=\textwidth,height=0.8\textwidth]{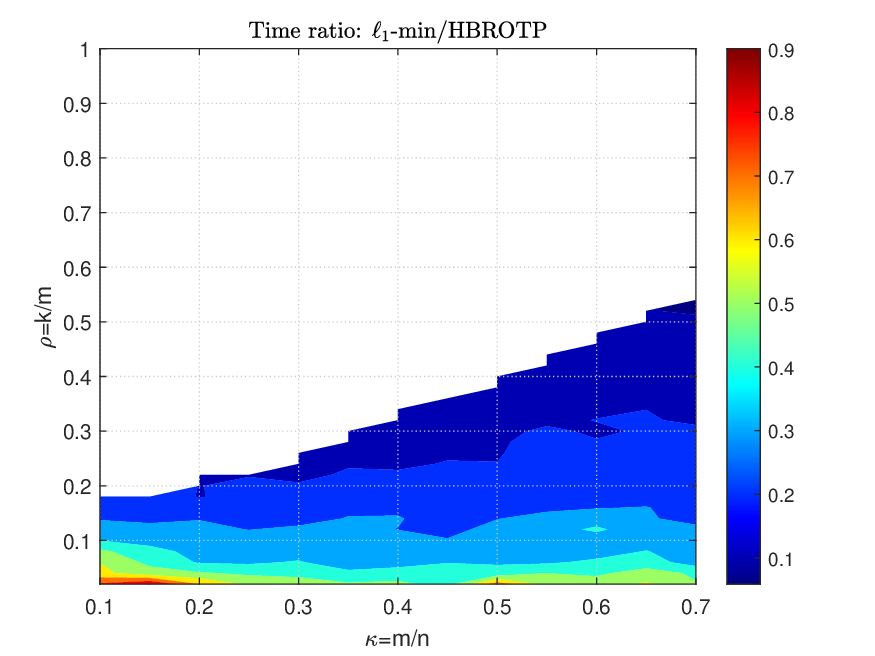}
  \end{minipage}
    }
   \subfigure[Time ratio of OMP and HBROTP]{
   \begin{minipage}[t]{0.45\linewidth}
  \centering
 \includegraphics[width=\textwidth,height=0.8\textwidth]{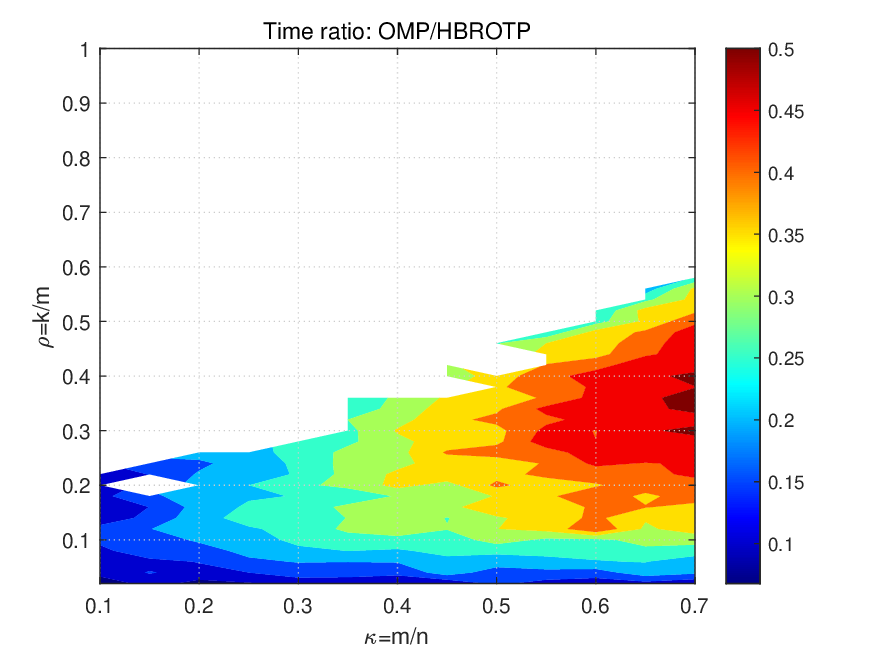}
    \end{minipage}
    }
   \caption{The ratios of average CPU time of  the algorithms.  }\label{fig-time-ratio}
   \end{figure}

 In the  intersection of the recovery regions of multiple algorithms, we compare the average CPU  time for signal recovery via these algorithms. Specifically, for each given $\kappa$, we test 10 problem instances for each algorithm with the mesh $(\kappa,\rho)$, wherein $\rho$  is ranged from 0.02 to 1 with stepsize 0.02 until the success rate of recovery is less than 90\%. The ratios of the average computational time of ROTP2, PGROTP, $\ell_1$-min and OMP against that of HBROTP are displayed  in Fig. \ref{fig-time-ratio} (a)-(d), respectively.  Fig. \ref{fig-time-ratio} (a) and (b)  show that HBROTP is at least 1.6 times faster than ROTP2  in most areas and slower than PGROTP except in the region  $[0.1,0.2]\times[0.02,0.1]$. On the other hand, from Fig. \ref{fig-time-ratio} (a)-(d), we observe that the ROT-type algorithms including HBROTP, ROTP2 and  PGROTP take relatively more time to solve the problems than $\ell_1$-min and OMP, due to solving quadratic convex optimization problems.

\subsection{Image reconstruction}\label{real_image}
 In this section, we compare the performances of several algorithms  on the  reconstruction of several images (\emph{Lena}, \emph{Peppers} and \emph{Baboon}) of size  $512 \times512$.  Only  accurate measurements are used in the experiment, and the measurement matrices are $m\times n$ normalized standard Gaussian matrices with $n=512$ and $m=\lceil\kappa\cdot n\rceil,$ where $\kappa$ is the sampling rate. The discrete wavelet transform with the \emph{`sym8'} wavelet is used to establish the sparse representation of the images. The input sparsity level is set as $k=\lceil n/10\rceil$ for  HBROTP, ROTP2  and PGROTP, and the   parameters of HBROTP are set as $\alpha=5$ and $ \beta = 0.2$. The peak signal-to-noise ratio (PSNR) is used to compare the reconstruction quality of images, which is  defined by
$$
PSNR:=10\cdot log_{10}(V^2/MSE),
$$
where $MSE$ denotes the mean-squared error  between the reconstructed  and  original image, and $V$ represents  the maximum fluctuation in the original image data type ($V=255$ is used in our experiments). Clearly, the larger the value of PSNR, the higher the reconstruction quality.
 \begin{table}[h]
 \centering
 \caption{Comparison of PSNR (dB) for  algorithms with different sampling rates.}\label{table-PSNR}
 \vspace{0.2cm}
 \begin{tabular}{|c|c|c|c|c|c|c|}
\hline
&$\kappa$& HBROTP & ROTP2 &  PGROTP &  $\ell_1$-min &   OMP\\
\hline
               &0.3 &32.60 	&32.34 	&33.12 	&33.63 	&31.37\\
Lena      &0.4 &34.37 	&32.49 	&31.75 	&35.10 	&32.95\\
              &0.5 &35.63 	&33.11 	&31.93 	&37.04 	&34.34\\
\hline
                 &0.3&31.31 	&32.33 	&33.27 	&33.03 	&30.17\\
Peppers  &0.4&33.10 	&31.78 	&31.60 	&34.08 	&31.66\\
                 &0.5&34.23 	&32.04 	&31.10 	&35.90 	&33.38\\
 \hline
                &0.3&28.70 	&31.35 	&32.33 	&29.90 	&28.35\\
 Baboon &0.4&29.12 	&30.06 	&30.00 	&30.05 	&28.53\\
                &0.5&29.37 	&30.06 	&30.07 	&30.20 	&28.78\\
\hline
 \end{tabular}
 \end{table}

The results in terms of PSNR with sampling rates $\kappa=0.3,0.4,0.5$ are summarized in Tab. \ref{table-PSNR}, from which we see that HBROTP is always superior to OMP and  inferior to $\ell_1$-min in reconstruction quality. For ROTP-type algorithms with $\kappa=0.4, 0.5$, the PSNR values of HBROTP exceed that of  ROTP2, PGROTP at least 1.88 dB for \emph{Lena} and 1.32 dB for \emph{Peppers}, respectively. In other cases, ROTP2 and PGROTP obtained better results than HBROTP in reconstruction quality. In particular, the performances of ROTP2 and PGROTP are always equivalent or superior to $\ell_1$-min for \emph{Baboon}.  In the meantime, the comparison of visual quality for the constructed images by HBROTP with $\kappa=0.3,0.4,0.5$ is displayed in Fig. \ref{Image}. It can be seen  that the reconstruction  quality has been significantly improved for three images as the sampling rate $\kappa$ is ranged from $0.3$ to $0.5$, and the best visual results have been achieved around $\kappa=0.5$.
\begin{figure}[h]
\subfigbottomskip=-0.5cm
\subfigcapskip=-0.5cm
   \centering
  \subfigure[$\begin{array}{c}\textrm{Lena}\\
  \textrm{(Original)}
  \end{array}$]{
   \begin{minipage}[t]{0.28\linewidth}
  \includegraphics[width=\textwidth,height=\textwidth]{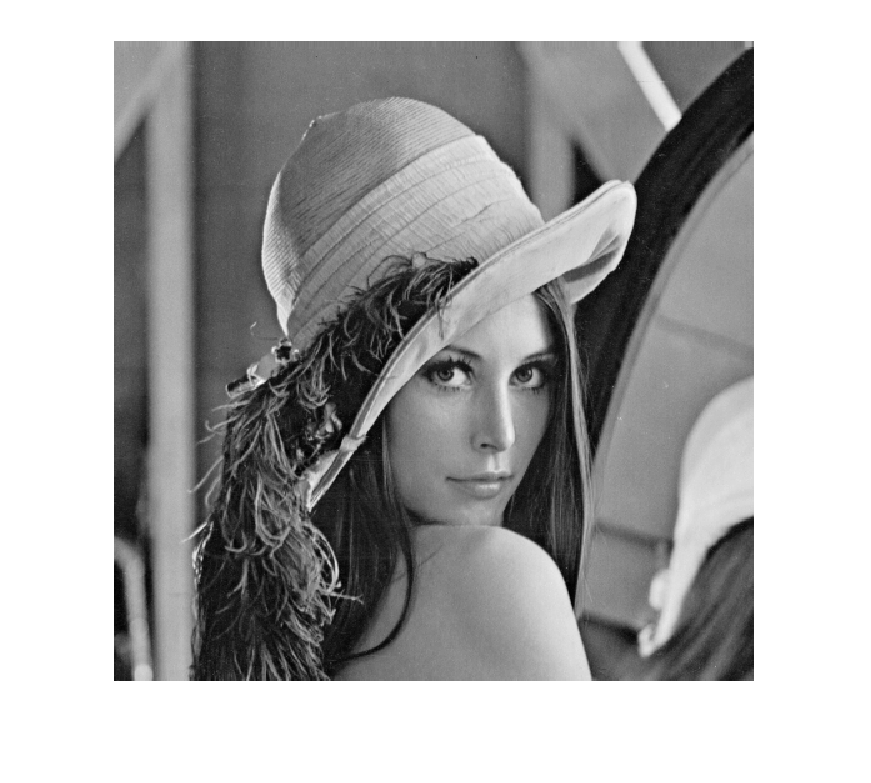}
  \end{minipage}
    } \hspace{-1cm}
     \subfigure[$\kappa=0.3$]{
   \begin{minipage}[t]{0.28\linewidth}
  \includegraphics[width=\textwidth,height=\textwidth]{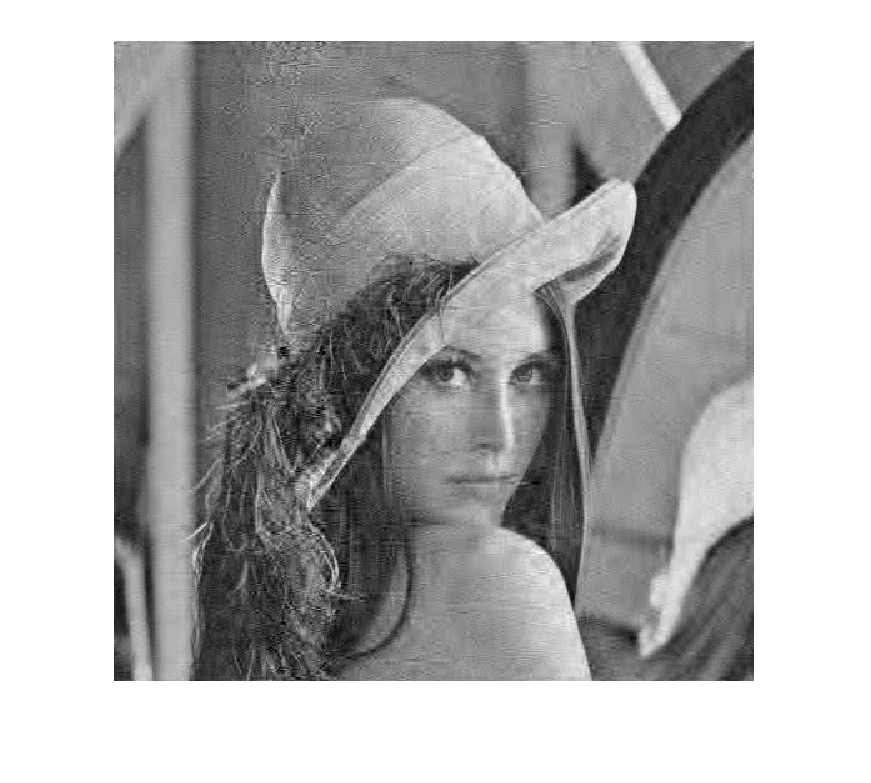}
  \end{minipage}
    }   \hspace{-1cm}
     \subfigure[$\kappa=0.4$]{
   \begin{minipage}[t]{0.28\linewidth}
  \includegraphics[width=\textwidth,height=\textwidth]{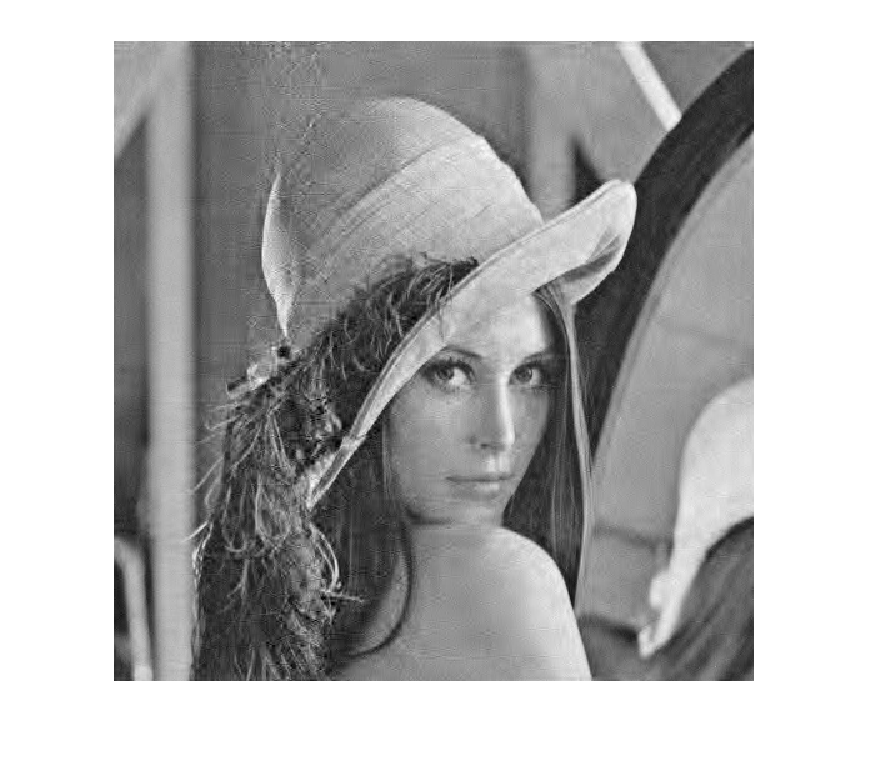}
  \end{minipage}
    }   \hspace{-1cm}
      \subfigure[$\kappa=0.5$]{
   \begin{minipage}[t]{0.28\linewidth}
  \includegraphics[width=\textwidth,height=\textwidth]{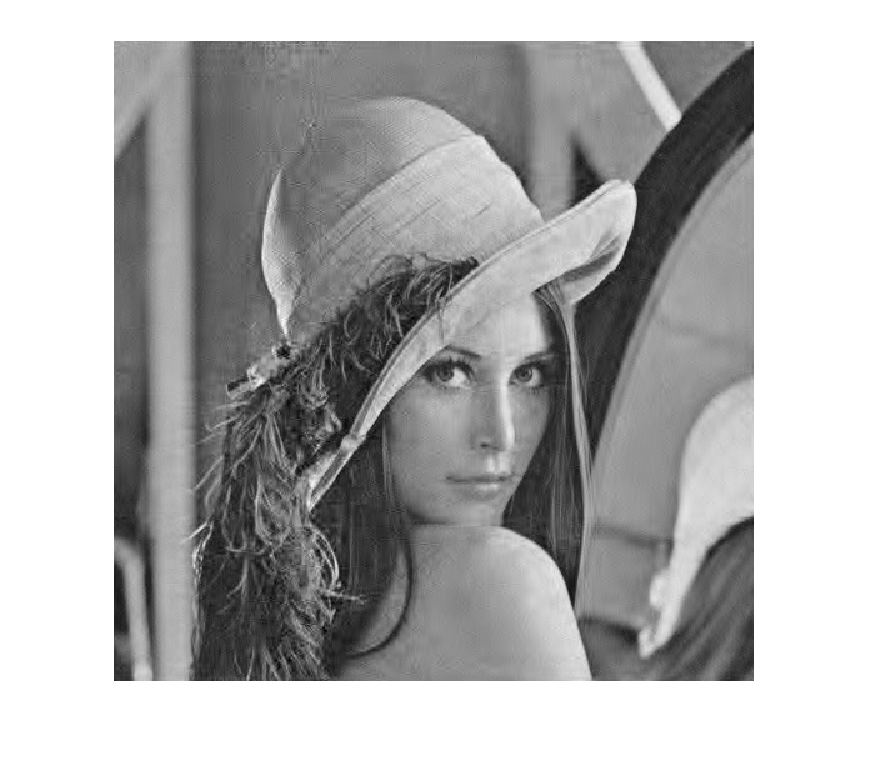}
  \end{minipage}
    }
 \vskip 0.05in
      \subfigure[$\begin{array}{c}\textrm{Peppers}\\
  \textrm{(Original)}
  \end{array}$]{
   \begin{minipage}[t]{0.28\linewidth}
  \includegraphics[width=\textwidth,height=\textwidth]{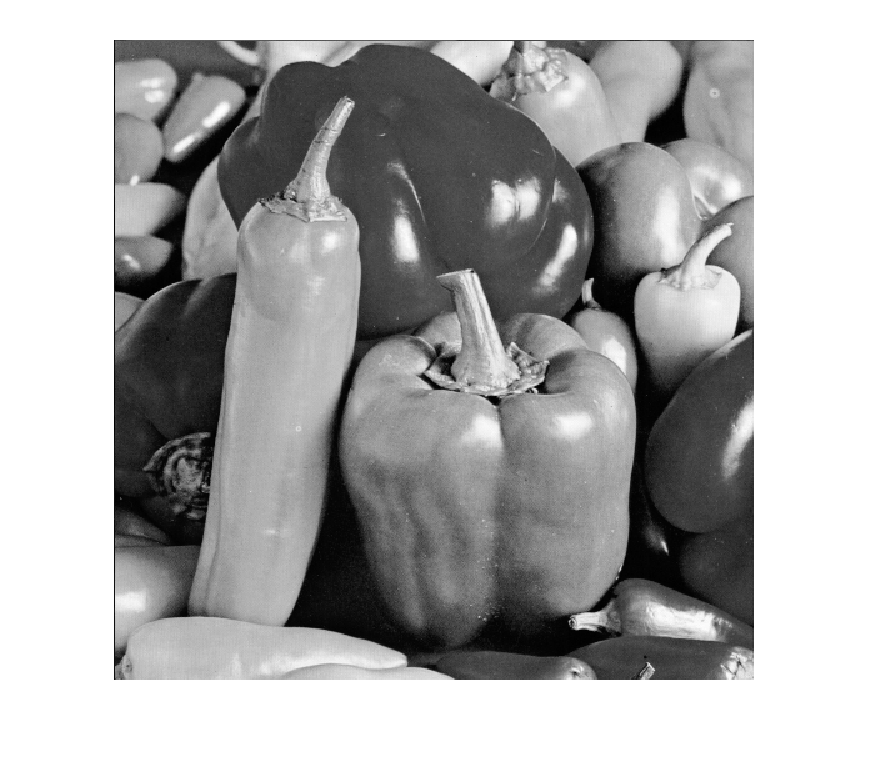}
  \end{minipage}
    } \hspace{-1cm}
     \subfigure[$\kappa=0.3$]{
   \begin{minipage}[t]{0.28\linewidth}
  \includegraphics[width=\textwidth,height=\textwidth]{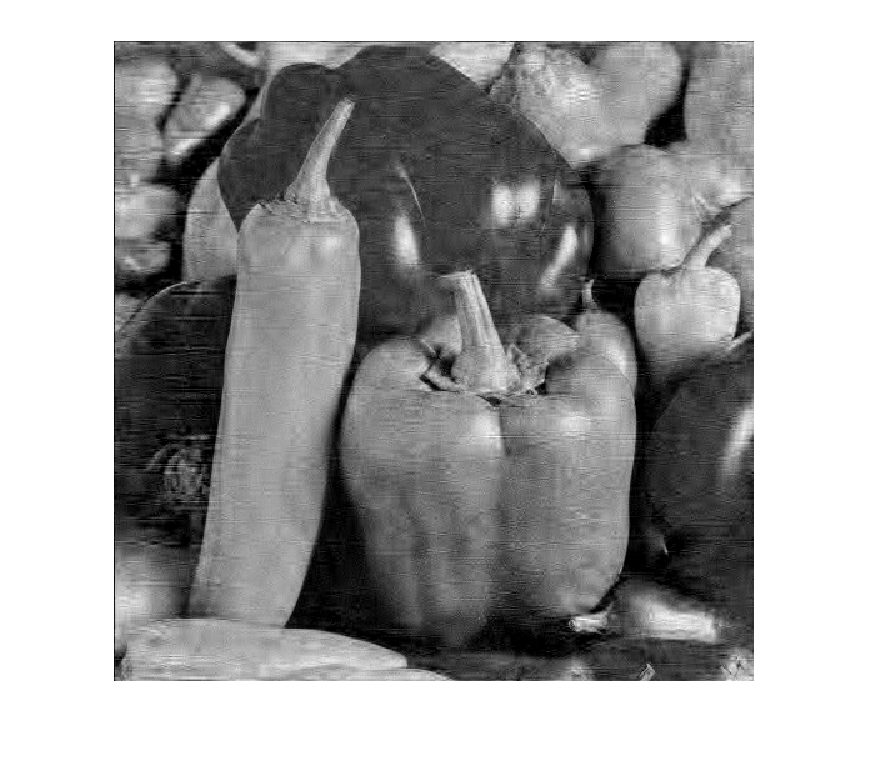}
  \end{minipage}
    }   \hspace{-1cm}
     \subfigure[$\kappa=0.4$]{
   \begin{minipage}[t]{0.28\linewidth}
  \includegraphics[width=\textwidth,height=\textwidth]{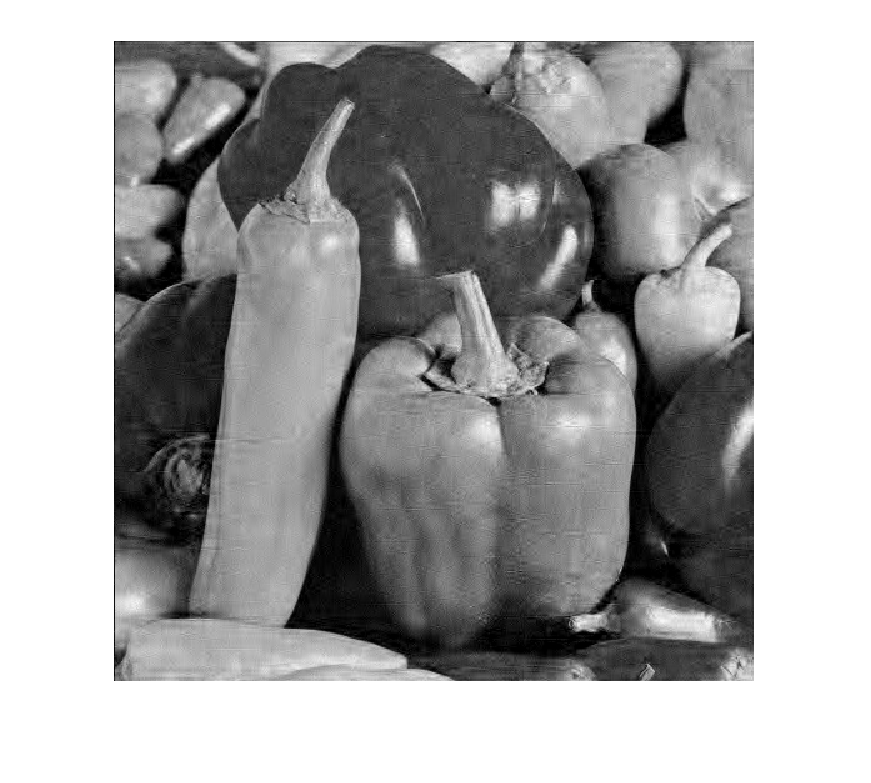}
  \end{minipage}
    }   \hspace{-1cm}
      \subfigure[$\kappa=0.5$]{
   \begin{minipage}[t]{0.28\linewidth}
  \includegraphics[width=\textwidth,height=\textwidth]{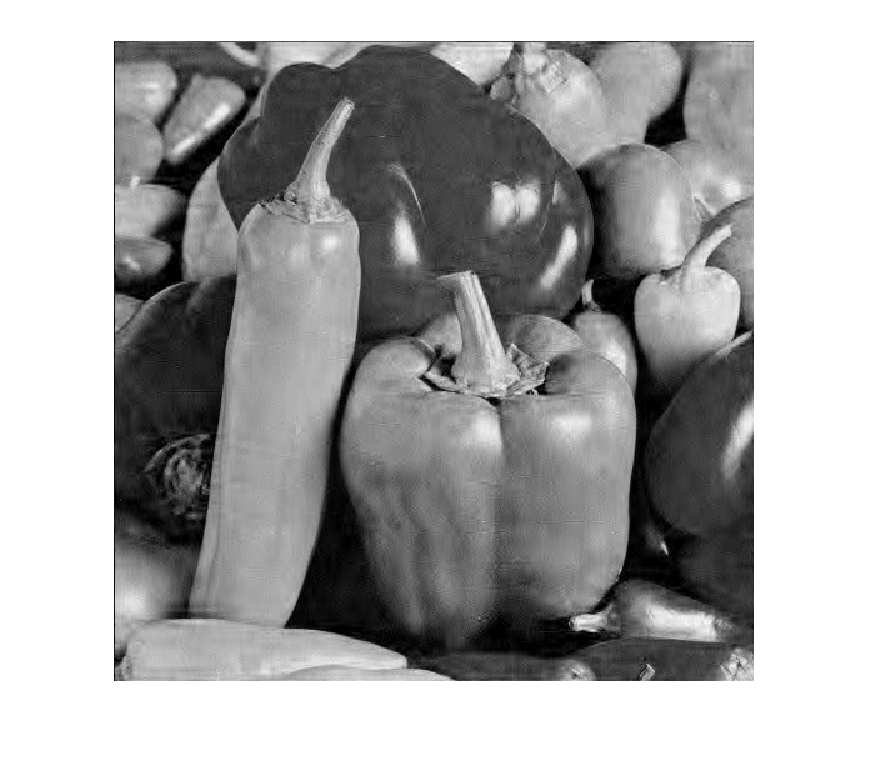}
  \end{minipage}
    }
 \vskip 0.05in
  \subfigure[$\begin{array}{c}\textrm{Baboon}\\
  \textrm{(Original)}
  \end{array}$]{
   \begin{minipage}[t]{0.28\linewidth}
  \includegraphics[width=\textwidth,height=\textwidth]{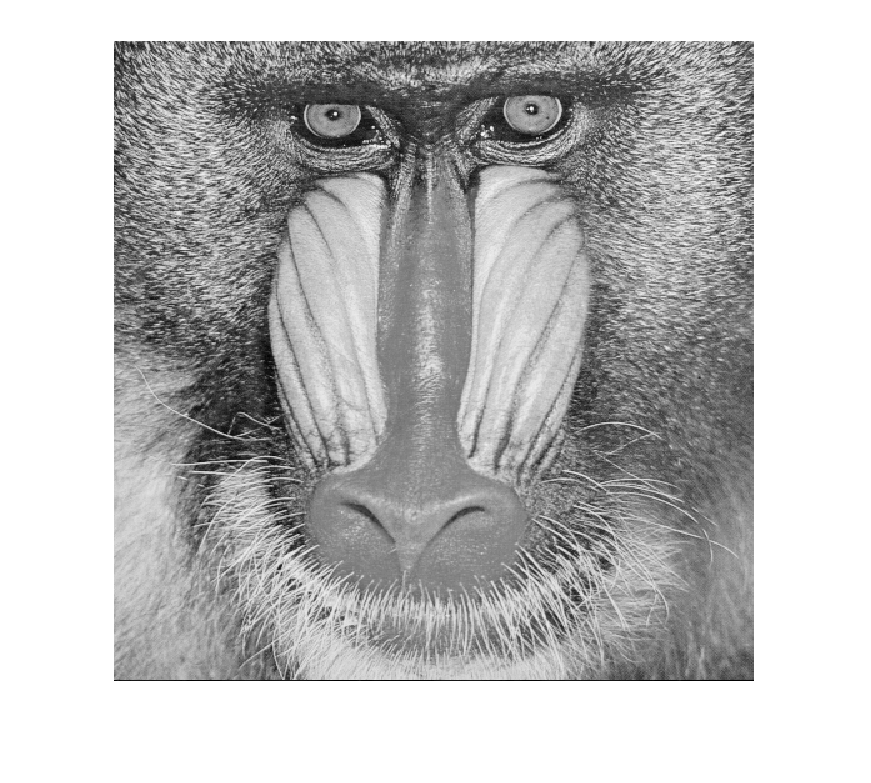}
  \end{minipage}
    } \hspace{-1cm}
     \subfigure[$\kappa=0.3$]{
   \begin{minipage}[t]{0.28\linewidth}
  \includegraphics[width=\textwidth,height=\textwidth]{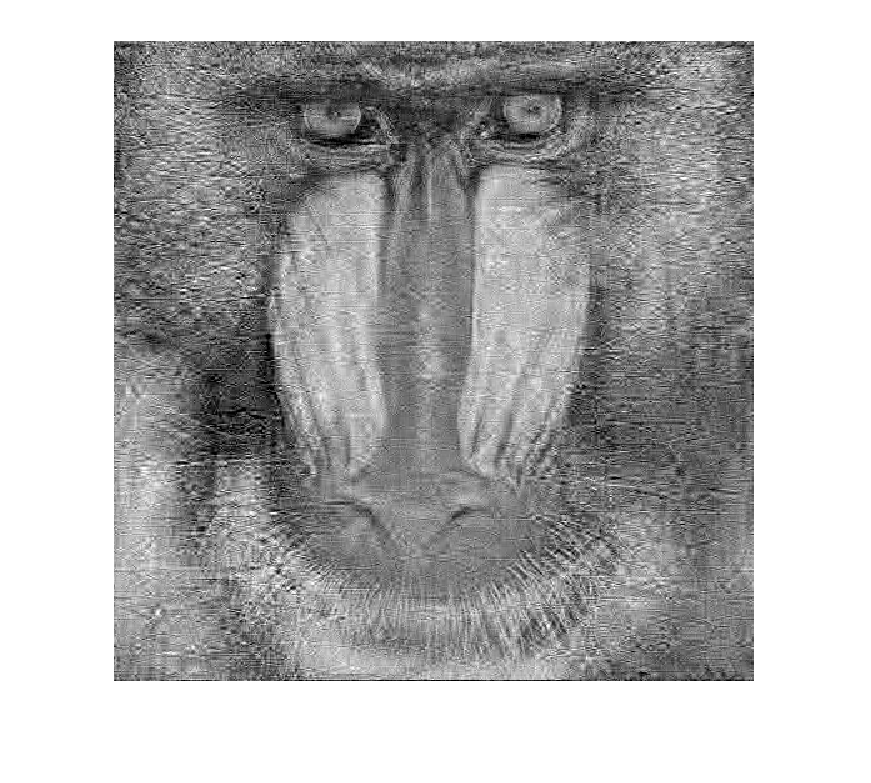}
  \end{minipage}
    }   \hspace{-1cm}
     \subfigure[$\kappa=0.4$]{
   \begin{minipage}[t]{0.28\linewidth}
  \includegraphics[width=\textwidth,height=\textwidth]{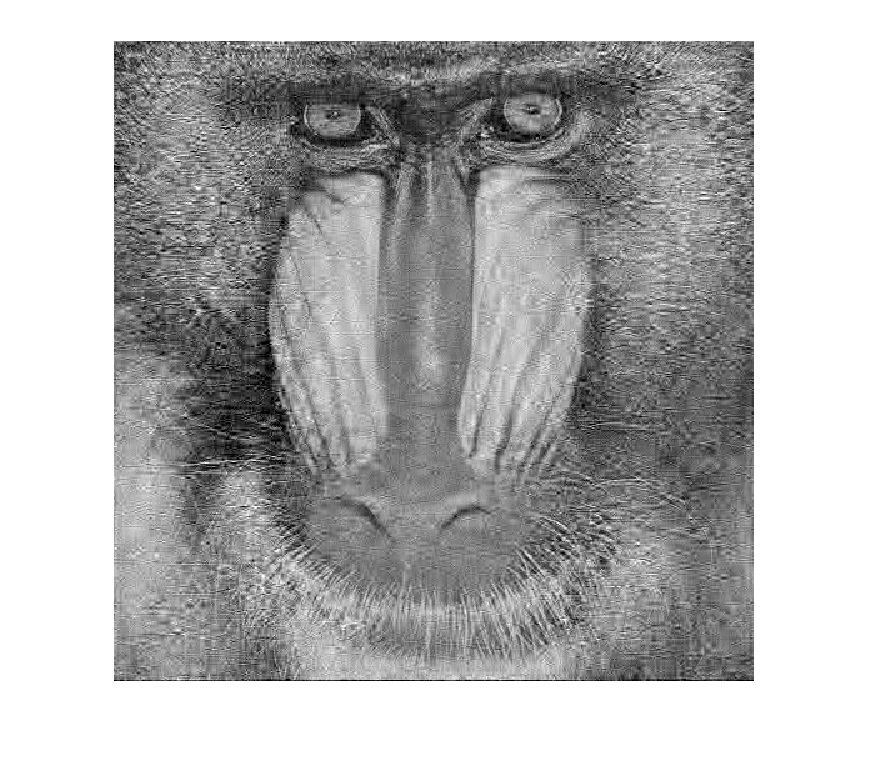}
  \end{minipage}
    }   \hspace{-1cm}
      \subfigure[$\kappa=0.5$]{
   \begin{minipage}[t]{0.28\linewidth}
  \includegraphics[width=\textwidth,height=\textwidth]{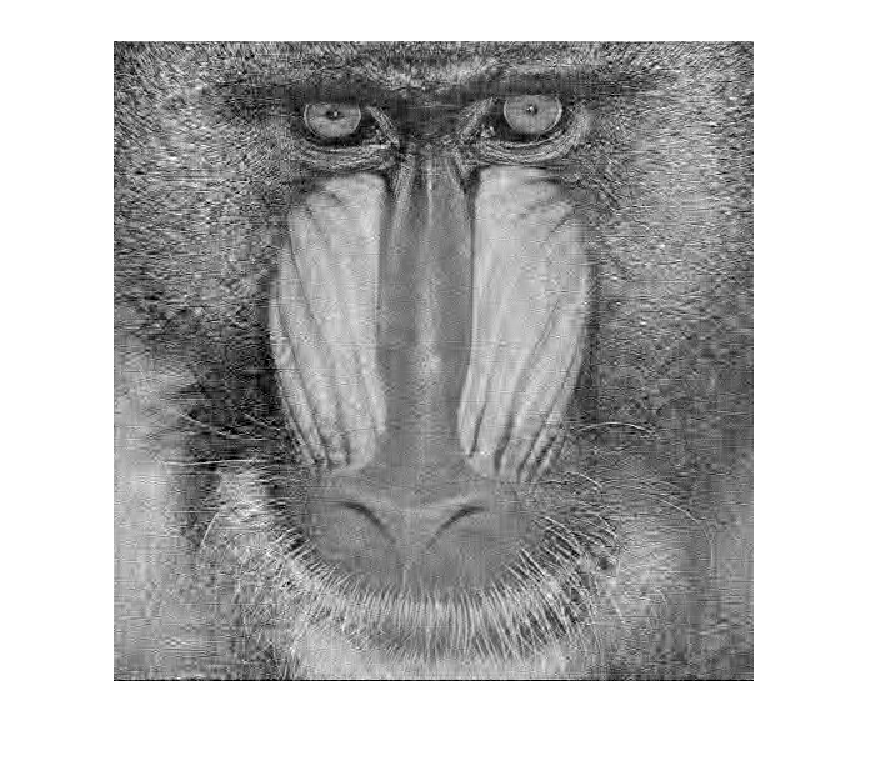}
  \end{minipage}
    }
   \caption{Performance of HBROTP for  three images with  different sampling rates. }\label{Image}
   \end{figure}

\subsection{Image deblurring and denoising}\label{deblur_image}
 In this section, we compare the performances of HBROTP and  PLB on  image deblurring and denoising.  In our experiments, several images including \emph{Boats}, \emph{Cameraman}, \emph{Clock}, \emph{Goldhill} and \emph{Shepp-Logan}  of size  $128 \times128$ are expressed as vectors in $\mathbb{R}^{n}$ with $n=16384$ through concatenating their columns.  For a given image $z$,  the corresponding blurred noisy image $y\in\mathbb{R}^{n}$ is obtained by  \eqref{model-lip}, in which  $\Phi\in\mathbb{R}^{n\times n}$  is the blurring matrix generated by a Gaussian kernel \emph{fspecial(`Gaussian',11,0.6)} in Matlab  with periodic boundary condition (see Chapter 4 in \cite{GL13}), and $\nu$  is a  Gaussian white noise vector  with mean 0 and standard deviation $\hat{\sigma}$. The sparse representation of $z$  is expressed as  $z=\Psi x$, where $\Psi$ is taken as  the synthesis operator generated by the linear B-splines \cite{ABR23,BPR21}, denoted $\Psi_1$, or the discrete wavelet matrix generated by the \emph{`sym8'} wavelet, denoted $\Psi_2$. Thus the image deblurring and denoising can be achieved by solving the corresponding SLI problem \eqref{model-slip}.

For  HBROTP, the  discrete wavelet transform, i.e., $\Psi=\Psi_2$ is used to achieve  the sparse
representation of the image,  and the   parameters in this algorithm are set as $k=\lceil 0.4n\rceil$, $\alpha=1$ and $\beta=0.8$. For PLB,  we use PLB$_i$  to represent PLB with $\Psi=\Psi_i$ for $i=1,2$, and  the  parameters $(\mu,d,\delta)$ are given as follows:   $\mu=0.05$ is determined experimentally in terms of PSNR;  the dimension of Krylov subspace is set as $d=11$ according to the suggestion in \cite{ABR23};   $\delta$ is the same as that of \cite{BPR21}.  The stopping criterion of the algorithm is given by
$$
\|x^{p+1}-x^p\|_2/\|x^{p+1}\|_2\leq 10^{-4}.
$$

 \begin{table}[h]
 \centering
 \caption{Comparison of PSNR (dB) and CPU time (in seconds) of HBROTP and PLB  on image deblurring and denoising with different standard deviation $\hat{\sigma}$.}\label{table-deblur-PSNR2}
 \vspace{0.2cm}
 \begin{tabular}{|c|c|c|c|c|c|c|c|}
 \hline
Standard&\multirow{2}*{Images}&\multicolumn{3}{|c|}{PSNR(dB)}&\multicolumn{3}{|c|}{CPU time(seconds)}\\
\cline{3-8} deviation& &PLB$_1$&PLB$_2$& HBROTP &PLB$_1$&PLB$_2$& HBROTP\\
 \hline
\multirow{6}*{$\hat{\sigma}=2$}&Barbara&35.34 	&35.34 	&37.75 		       &0.80	&4.31	&2291\\
&Boats&35.53 	&35.52 	&35.34 		           &0.83	&3.30	&2049\\
&Cameraman&35.58 	&35.57 	&37.26 		&0.83	&2.80	&1383\\
&Clock&35.66 	&35.65 	&38.12 		          &0.59	&2.13	&1305\\
&Goldhill&35.34 	&35.33 	&35.05 	        	&0.70	&2.81	&2319\\
&Shepp-Logan&35.54 	&35.51 	&38.72 		&0.86	&3.28	&1764\\
\hline
\multirow{6}*{$\hat{\sigma}=4$}&Barbara&30.74 	&30.74 	&33.86 		&0.94	&4.03	&2071\\
&Boats&30.80 	&30.80 	&33.06 		&0.38	&3.11	&2411\\
&Cameraman&30.79 	&30.79 	&33.53 		&0.78	&4.14	&1370\\
&Clock&30.90 	&30.90 	&33.44 		&0.53	&3.28&	1737\\
&Goldhill&30.76 	&30.75 	&33.02 		&0.80	&3.95	&2444\\
&Shepp-Logan&30.93 	&30.92 	&33.70 		&0.61	&4.66	&1381\\
\hline
 \end{tabular}
 \end{table}

The results in terms of CPU time and PSNR  for HBROTP and PLB on   image deblurring and denoising with two different standard deviations $\hat{\sigma}=2, 4$ are given in Tab. \ref{table-deblur-PSNR2}. In the case   $\hat{\sigma}=2$,  the PSNR values of HBROTP exceed that of PLB$_1$ and PLB$_2$ at least 1.6 dB for all images except \emph{Boats} and \emph{Goldhill}. As the  noise intensity increases,  the differences of   PSNR values between HBROTP and  PLB$_i(i=1,2)$ are enlarged to 2.2 dB for all images as $\hat{\sigma}=4$. This experiment shows that  HBROTP can be stronger than  PLB  on image deblurring and denoising, and  HBROTP is more stable than PLB in noisy situations. However, solving quadratic subproblem \eqref{algorithm-ROTP-2} causes the HBROTP method to consume  more time than PLB$_1$ and PLB$_2. $  Moreover, PLB$_1$ is  faster than PLB$_2$ since the synthesis operator $\Psi_1$ is more sparser than the discrete wavelet matrix $\Psi_2$. Finally, the deblurring/denoising effects of HBROTP and  PLB$_1$ on \emph{Cameraman} and \emph{Shepp-Logan} with $\hat{\sigma}=2$ are shown in Fig. \ref{deblur_Image}, from which  it can be observed that both HBROTP and PLB$_1$ can successfully recover the two images in high quality.

\begin{figure}[h]
   \centering
  \subfigure[$\begin{array}{c}\textrm{Cameraman}\\
  \textrm{(Original)}
  \end{array}$]{
   \begin{minipage}[t]{0.22\linewidth}
  \includegraphics[width=\textwidth,height=\textwidth]{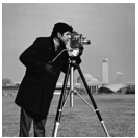}
  \end{minipage}
    } 
     \subfigure[$\begin{array}{c}\textrm{Noisy Blurred}\\
  \textrm{ Image}
  \end{array}$]{
   \begin{minipage}[t]{0.22\linewidth}
  \includegraphics[width=\textwidth,height=\textwidth]{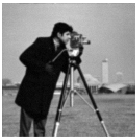}
  \end{minipage}
    }  
     \subfigure[PLB$_1$]{
   \begin{minipage}[t]{0.22\linewidth}
  \includegraphics[width=\textwidth,height=\textwidth]{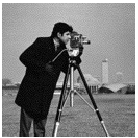}
  \end{minipage}
    }  
      \subfigure[HBROTP]{
   \begin{minipage}[t]{0.22\linewidth}
  \includegraphics[width=\textwidth,height=\textwidth]{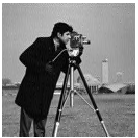}
  \end{minipage}
    }
 \vskip 0.05in
   \subfigure[$\begin{array}{c}\textrm{Shepp-Logan}\\
  \textrm{(Original)}
  \end{array}$]{
   \begin{minipage}[t]{0.22\linewidth}
  \includegraphics[width=\textwidth,height=\textwidth]{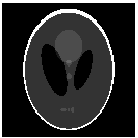}
  \end{minipage}
    } 
     \subfigure[$\begin{array}{c}\textrm{Noisy Blurred}\\
  \textrm{ Image}
  \end{array}$]{
   \begin{minipage}[t]{0.22\linewidth}
  \includegraphics[width=\textwidth,height=\textwidth]{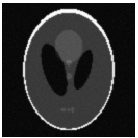}
  \end{minipage}
    }  
     \subfigure[PLB$_1$]{
   \begin{minipage}[t]{0.22\linewidth}
  \includegraphics[width=\textwidth,height=\textwidth]{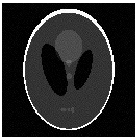}
  \end{minipage}
    }  
      \subfigure[HBROTP]{
   \begin{minipage}[t]{0.22\linewidth}
  \includegraphics[width=\textwidth,height=\textwidth]{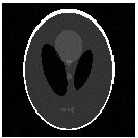}
  \end{minipage}
    }
   \caption{Performance of PLB$_1$ and HBROTP on image deblurring and denoising with $\hat{\sigma}=2$. }\label{deblur_Image}
   \end{figure}

\section{Conclusions}\label{conclusion}
The new algorithms that combine the optimal $k$-thresholding and heavy-ball technique are proposed in this paper. Such algorithms can be seen as the acceleration versions of the optimal $k$-thresholding methods. The solution error bounds and convergence of the proposed algorithms have been shown mainly under the RIP of the matrices. The numerical performance of the proposed HBROTP algorithm has been evaluated through phase transition,  average runtime and image processing. The experiment results indicate that HBROTP  is a robust signal recovery method, especially when the sampling rate is relatively low (e.g., $\kappa \leq 0.5$), and it is generally faster than the standard ROTP method thank to the heavy-ball acceleration technique.\\




\end{document}